\documentclass{amsart}
\usepackage{amsxtra}
\usepackage{mathdots}
\usepackage{mathrsfs}
\usepackage{verbatim}
\usepackage{calc}
\usepackage{graphicx}
\usepackage{varwidth}
\usepackage{hyperref}
\usepackage[shortalphabetic,initials]{amsrefs}
\usepackage{bdsstandard}

\numberwithin{equation}{subsection}

\theoremstyle{plain}
\newtheorem{thm}[equation]{Theorem}
\newtheorem*{thm*}{Theorem}
\newtheorem{prop}[equation]{Proposition}
\newtheorem*{prop*}{Proposition}
\newtheorem{cor}[equation]{Corollary}
\newtheorem*{cor*}{Corollary}
\newtheorem{lem}[equation]{Lemma}
\newtheorem*{lem*}{Lemma}

\newtheorem*{conj*}{Conjecture}

\theoremstyle{definition}

\newtheorem*{defn*}{Definition}

\newtheorem*{eg*}{Example}

\newtheorem*{ex*}{Exercise}
\newtheorem{rk}[equation]{Remark}
\newtheorem*{rk*}{Remark}

\newtheorem*{ntn*}{Notation}

\theoremstyle{plain}
\newtheorem*{mthm}{Main Theorem}

\renewcommand{\L}{\ensuremath{\mathscr{L}}\xspace}

\DeclareMathOperator{\Adm}{Adm}
\DeclareMathOperator{\Conv}{Conv}
\DeclareMathOperator{\Perm}{Perm}

\newcommand{\aff}{\ensuremath{\mathrm{a}}\xspace}
\newcommand{\loc}{\ensuremath{\mathrm{loc}}\xspace}
\newcommand{\naive}{\ensuremath{\mathrm{naive}}\xspace}
\newcommand{\spin}{\ensuremath{\mathrm{spin}}\xspace}

\entrymodifiers={+!!<0pt,\fontdimen22\textfont2>}

\begin{document}

\renewcommand{\O}{\ensuremath{\mathscr{O}}\xspace}

\title[Topological flatness of ramified unitary local models. I.]{Topological flatness of local\\ models for ramified unitary groups.\\  I.  The odd dimensional case}
\author{Brian D. Smithling}
\address{University of Toronto, Department of Mathematics, 40 St.\ George St.,\ Toronto, ON  M5S 2E4, Canada}
\email{bds@math.toronto.edu}

\begin{abstract}
Local models are certain schemes, defined in terms of linear-alge\-bra\-ic moduli problems, which give \'etale-local neighborhoods of integral models of certain $p$-adic PEL Shimura varieties defined by Rapoport and Zink.  When the group defining the Shimura variety ramifies at $p$, the local models (and hence the Shimura models) as originally defined can fail to be flat, and it becomes desirable to modify their definition so as to obtain a flat scheme.  In the case of unitary similitude groups whose localizations at $\QQ_p$ are ramified, quasi-split $GU_n$, Pappas and Rapoport have added new conditions, the so-called wedge and spin conditions, to the moduli problem defining the original local models and conjectured that their new local models are flat.  We prove a preliminary form of their conjecture, namely that their new models are topologically flat, in the case $n$ is odd.
\end{abstract}
\maketitle

\section{Introduction}
A basic problem for a given Shimura variety is to define a reasonable model for the variety over the ring of integers of its reflex field.  In \cite{rapzink96}, Rapoport and Zink define natural models for certain PEL Shimura varieties with parahoric level structure at $p$ over the ring of integers in the completion of the reflex field at any place dividing $p$.  In addition, they reduce many aspects of the study of these models to the associated \emph{local models;} these give \'etale-local neighborhoods of the Rapoport-Zink models which are defined in terms of purely linear-algebraic moduli problems, and thus --- at least in principle --- are more amenable to direct investigation.

There has been much study of local models in various cases since the appearance of Rapoport and Zink's book; see, for example, work of Pappas \cite{pap00}, G\"ortz \cites{goertz01,goertz03,goertz04,goertz05}, Haines and Ng\^o \cite{hngo02b}, Pappas and Rapoport \cites{paprap03,paprap05,paprap08,paprap09}, Kr\"amer \cite{kr03}, Arzdorf \cite{arz09}, and the author \cite{sm09a}.  Beginning with Pappas's paper \cite{pap00}, it has come to be understood that when the group defining the Shimura variety is \emph{ramified} at $p$ --- that is, the base change of the group to $\QQ_p$ splits only after passing to a ramified extension of $\QQ_p$ --- then the associated model and local model need not satisfy one of the most basic criteria for reasonableness, namely they need not be flat.  In such cases, the models and local models defined by Rapoport and Zink have come to be renamed \emph{naive} models and local models, respectively, with the true models and local models defined as the scheme-theoretic closures of the generic fibers in the respective naive models and local models.

Sparked by Pappas's observation, it has become an important part of the subject to better understand local models in the presence of ramification;
see \cites{goertz05,pap00,paprap03,paprap05,paprap08,paprap09}.  In particular, it is an interesting problem to obtain a \emph{moduli-theoretic description} of the local model in such cases, ideally by refining the moduli problem that describes the naive local model.  This is the problem of concern in this paper and its sequel \cite{sm10a}, in one of the basic cases in which ramification arises:  a unitary similitude group attached to an imaginary quadratic number field ramified at $p\neq 2$.

More precisely, let $F/F_0$ be a ramified quadratic extension of discretely valued, non-Archimedean fields of residual characteristic not $2$.  Endow $F^n$, $n \geq 3$, with the $F/F_0$-Hermitian form $\phi$ specified by the values $\phi(e_i,e_j) = \delta_{i,n+1-j}$ on the standard basis vectors $e_1,\dotsc$, $e_n$, and consider the reductive group $GU_{n}:= GU(\phi)$ over $F_0$.  In this paper we consider exclusively the case that $n$ is odd; the case of even $n$ will be treated in \cite{sm10a}.  In the odd case, every parahoric subgroup of $GU_n(F_0)$ can be described as the stabilizer of a self-dual periodic lattice chain in $F^n$, and the conjugacy classes of parahoric subgroups can be naturally parametrized by the nonempty subsets of $\{0,\dotsc,m\}$, where $n = 2m+1$; see \eqref{st:parahoric_description} or, for more details, \cite{paprap09}*{\s1.2.3(a)}.  Identifying $G \tensor_{F_0} F \iso GL_{n,F} \times \GG_{m,F}$, let $\mu_{r,s}$ denote the cocharacter $\bigl(1^{(s)},0^{(r)},1\bigr)$ of $D \times \GG_{m,F}$, where $D$ denotes the standard maximal torus of diagonal matrices in $GL_{n,F}$ and $r$ and $s$ are nonnegative integers with $r + s = n$.  In the special case that $F_0 = \QQ_p$ and $(F^n,\phi)$ is isomorphic to the $\QQ_p$-localization of a Hermitian space $(K^n,\psi)$ with $K$ an imaginary quadratic number field, the pair $(r,s)$ denotes the signature of $(K^n,\psi)$ and $\mu_{r,s}$ is a cocharacter obtained in the usual way from the Shimura datum attached to the associated unitary similitude group, as in \cite[\s1.1]{paprap09}.


For nonempty $I \subset \{0,\dotsc,m\}$, we may then consider the \emph{naive local model $M_I^\naive$} in the sense of \cite{rapzink96}; this is a projective $\O_F$-scheme whose explicit definition we recall in \s\ref{ss:naive_lm}.  As \cite{pap00} observes, the naive local model fails to be flat over $\O_F$ in general.  
In response, the papers \cites{pap00,paprap09} attempt to correct for non-flatness by adding new conditions to the moduli problem defining $M_I^\naive$:  
first Pappas adds the \emph{wedge condition} to define a closed subscheme $M_I^\wedge \subset M_I^\naive$, the \emph{wedge local model} (see \s\ref{ss:wedge_cond}); and then Pappas and Rapoport add a further condition, the \emph{spin condition}, to define a third closed subscheme $M_I^\spin \subset M_I^\wedge$, the \emph{spin local model} (see \s\ref{ss:spin_cond}).  The schemes $M_I^\naive$, $M_I^\wedge$, $M_I^\spin$, and the honest local model $M_I^\loc$ all have common generic fiber, and Pappas and Rapoport conjecture the following.

\begin{conj*}[Pappas-Rapoport \cite{paprap09}*{7.3}]
$M_I^\spin$ coincides with the local model $M_I^\loc$ inside $M_I^\naive$; or in other words, $M_I^\spin$ is flat over $\O_F$.
\end{conj*}

Although the conjecture remains open in general, Pappas and Rapoport have obtained a good deal of computer evidence in support of it.  The main result of this paper is essentially a preliminary form of the conjecture, which we state precisely as follows.

\begin{mthm}
The schemes  $M_I^\spin$ and $M_I^\wedge$ are topologically flat over $\O_F$; or in other words, the underlying topological spaces of $M_I^\spin$, $M_I^\wedge$, and $M_I^\loc$ coincide.
\end{mthm}

See \eqref{st:toplogical_agreement} and \eqref{st:top_flat}; recall that a scheme over a regular, integral, $1$-dimensional base scheme is \emph{topologically flat} if its generic fiber is dense.  The theorem notwithstanding, the scheme structures on $M_I^\wedge$ and $M_I^\spin$ really can differ, and it is only $M_I^\spin$ that is conjectured to be flat in general; see \cite{paprap09}*{7.4(iv)}.  We shall show in \cite{sm10a} that the spin local models $M_I^\spin$ attached to \emph{even} ramified unitary groups are also topologically flat; here the schemes $M_I^\wedge$ and $M_I^\spin$ usually do not even agree set-theoretically.

Following G\"ortz \cite{goertz01} (see also \cites{goertz03,goertz05,paprap03,paprap05,paprap08,paprap09,sm09a}), the key technique in the proof of the Main Theorem is to embed the special fiber of $M_I^\naive$ in an appropriate affine flag variety, where it and the special fibers of $M_I^\wedge$, $M_I^\spin$, and $M_I^\loc$ become stratified into finitely many Schubert cells.  Pappas and Rapoport show that the Main Theorem follows from showing that the Schubert cells in the special fibers of $M_I^\wedge$ and $M_I^\spin$ are indexed by the \emph{$\mu_{r,s}$-admissible set;} see \s\ref{ss:top_flat}.
We solve this problem by translating it into an equivalent one for $GSp_{2m}$, which in turn boils down to obtaining a concrete description of the admissible set for the cocharacter $\bigl(2^{(s)},1^{(2m-2s)},0^{(s)}\bigr)$ of $GSp_{2m}$.  Here we make key use of a result of Haines and Ng\^o \cite{hngo02a} that describes admissible sets for $GSp_{2m}$ in terms of \emph{permissible} sets for $GL_{2m}$.  As a byproduct of our considerations, we show that the notions of $\bigl(2^{(s)},1^{(2m-2s)},0^{(s)}\bigr)$-admissibility and $\bigl(2^{(s)},1^{(2m-2s)},0^{(s)}\bigr)$-permissibility for $GSp_{2m}$ are equivalent.

Finally, we remark that failure of flatness of the models in \cite{rapzink96} is not a phenomenon related solely to ramification:  as observed by Genestier, the local models in \cite{rapzink96} attached to split even orthogonal groups also fail to be flat in general.  See \cites{paprap09,sm09a}.

We now outline the contents of the paper.  Sections \ref{s:loc_mod}--\ref{s:afv} consist almost entirely of review from \cite{paprap09}:  in \s\ref{s:loc_mod} we review the definitions of the various local models, in \s\ref{s:GU} we review some group-theoretic aspects of ramified $GU_n$, and in \s\ref{s:afv} we review the embedding of the special fiber of the local model into an appropriate affine flag variety attached to $GU_n$.  In \s\ref{s:schub_cells} we obtain combinatorial descriptions of the Schubert cells contained in the special fibers of $M_{I}^\wedge$ and $M_{I}^\spin$ inside the affine flag variety, and we reduce the Main Theorem to showing that these cells are indexed by the $\mu_{r,s}$-admissible set.  In \s\ref{s:combinatorics} we solve this last problem, as described above.

\subsection*{Acknowledgments}
I thank Michael Rapoport for his encouragement to work on this problem and for introducing me to the subject.  I also thank him and Robert Kottwitz for reading and offering comments on a preliminary version of this article; and the referee for carefully reading the article and suggesting some corrections and improvements.  Part of the writing of this article was undertaken at the Hausdorff Research Institute for Mathematics in Bonn, which I thank for its hospitality and excellent working conditions.

\subsection*{Notation}
We fix once and for all an odd integer $n = 2m+1 \geq 3$ and a partition $n = s + r$ with $0 \leq s \leq m$, so that $s < r$.  Although almost everything we shall do will depend on $n$ and this partition, we shall usually not embed these choices into the notation.

We let $F/F_0$ denote a ramified quadratic extension of discretely valued, non-Archimedean fields with respective rings of integers $\O_F$ and $\O_{F_0}$, respective uniformizers $\pi$ and $\pi_0$ satisfying $\pi^2 = \pi_0$, and common residue field $k$ of characteristic not $2$.  We also employ an auxiliary ramified quadratic extension $K/K_0$ of discretely valued, non-Archimedean Henselian fields with respective rings of integers $\O_{K}$ and $\O_{K_0}$, respective uniformizers $u$ and $t$ satisfying $u^2 = t$, and the same residue field $k$; eventually $K$ and $K_0$ will be the fields of Laurent series $k((u))$ and $k((t))$, respectively.  We put $\Gamma:= \Gal(K/K_0)$, and we write $x \mapsto \ol x$ for the action of the nontrivial element of $\Gamma$ on $K$; then $\ol u = -u$.  Abusing notation, we continue to write $x \mapsto \ol x$ for the $R$-algebra automorphism of $K \tensor_{K_0} R$ induced by any base change $K_0 \to R$.

We relate objects by writing $\iso$ for isomorphic, $\ciso$  for canonically isomorphic, and $=$  for equal.

Given a vector $v \in \RR^l$, we write $v(j)$ for its $j$th entry, $\Sigma v$ for the sum of its entries, and $v^*$ for the vector in $\RR^l$ defined by $v^*(j) = v(l+1-j)$.  Given another vector $w$, we write $v \geq w$ if $v(j) \geq w(j)$ for all $j$.  We write $\mathbf d$ for the vector $(d,d,\dotsc,d)$, leaving it to context to make clear the number of entries.  The expression $\bigl(d^{(i)}, e^{(j)}, \dotsc\bigr)$ denotes the vector with $d$ repeated $i$ times, followed by $e$ repeated $j$ times, and so on.

We write $S_l$ for the symmetric group on $1,\dotsc,$ $l$, and $S_l^*$ for its subgroup
\[
   S_l^* := \bigl\{\, \sigma \in S_l \bigm| \sigma(l+1-j) = l+1 - \sigma(j)
                     \text{ for all } j \in \{1,\dotsc,l\}\,\bigr\}.
\]


\section{Unitary local models}\label{s:loc_mod}

We begin by recalling the definition and some of the discussion of local models for odd ramified unitary groups from \cite{paprap09}.  Let $I \subset \{0,\dotsc,m\}$ be a nonempty subset.

\subsection{Pairings}\label{ss:pairings}
Let $V := F^{n}$.  In this subsection we introduce some pairings on $V$ and notation related to them which we'll use throughout the article.

Let $e_1$, $e_2,\dotsc$, $e_{n}$ denote the standard ordered $F$-basis in $V$, and let
\[
   \phi\colon V \times V \to F
\]
be the $F/F_0$-Hermitian form on $V$ whose matrix with respect to the standard basis is
\begin{equation}\label{disp:antidiag_1}
   \begin{pmatrix}
     &  &  1\\
     & \iddots\\
     1
   \end{pmatrix}.
\end{equation}
We attach to $\phi$ the alternating and symmetric $F_0$-bilinear forms $V \times V \to F_0$ given respectively by
\begin{equation}\label{disp:pairings}
   \langle x,y \rangle := \tfrac 1 2 \Tr_{F/F_0}\bigl( \pi^{-1}\phi(x,y) \bigr)
   \quad\text{and}\quad
   (x,y) := \tfrac 1 2 \Tr_{F/F_0}\bigl( \phi(x,y) \bigr).
\end{equation}

For any $\O_F$-lattice $\Lambda \subset V$, we denote by $\wh \Lambda$ the $\phi$-dual of $\Lambda$,
\begin{equation}\label{disp:wh_Lambda}
   \wh \Lambda := \bigl\{\, x\in V \bigm| \phi(\Lambda,x) \subset \O_F \,\bigr\}.
\end{equation}
Then $\wh\Lambda$ is also the $\langle$~,~$\rangle$-dual of $\Lambda$,
\[
   \wh\Lambda = \bigl\{\,  x\in V \bigm| \langle \Lambda,x\rangle \subset \O_{F_0} \,\bigr\};
\]
and $\wh\Lambda$ is related to the $($~,~$)$-dual $\wh\Lambda^s := \{\,  x\in V \mid ( \Lambda,x ) \subset \O_{F_0} \,\}$ by the formula $\wh\Lambda^s = \pi^{-1}\wh\Lambda$.  Both $\wh\Lambda$ and $\wh\Lambda^s$ are $\O_F$-lattices in $V$, and the forms $\langle$~,~$\rangle$ and $($~,~$)$ induce perfect $\O_{F_0}$-bilinear pairings
\begin{equation}\label{disp:perf_pairing}
   \Lambda \times \wh\Lambda \xra{\text{$\langle$~,~$\rangle$}} \O_{F_0}
   \quad\text{and}\quad
   \Lambda \times \wh\Lambda^s \xra{\text{$($~,~$)$}} \O_{F_0}
\end{equation}
for all $\Lambda$.

\subsection{Standard lattices}\label{ss:lattices}

For $i = nb+c$ with $0 \leq c < n$, we define the \emph{standard $\O_F$-lattice}
\begin{equation}\label{disp:Lambda_i}
   \Lambda_i := \sum_{j=1}^c\pi^{-b-1}\O_F e_j + \sum_{j=c+1}^{n} \pi^{-b}\O_F e_j \subset V.
\end{equation}
Then $\wh\Lambda_i = \Lambda_{-i}$ for all $i$, and the $\Lambda_i$'s form a complete, periodic, self-dual lattice chain
\[
   \dotsb \subset \Lambda_{-2} \subset \Lambda_{-1} \subset \Lambda_0 \subset \Lambda_1 \subset \Lambda_2 \subset \dotsb,
\]
which we call the \emph{standard lattice chain}.  More generally, for any nonempty subset $I \subset \{0,\dotsc,m\}$, we denote by $\Lambda_I$ the periodic, self-dual subchain of the standard chain consisting of all lattices of the form $\Lambda_i$ for $i \in n\ZZ \pm I$.  Of course, in this way $\Lambda_{\{0,\dotsc,m\}}$ denotes the standard chain itself.

The standard lattice chain admits the following obvious trivialization.  Let $\epsilon_1,\dotsc$, $\epsilon_n$ denote the standard basis of $\O_F^n$, and let $\beta_i\colon \O_F^{n} \to \O_F^{n}$ multiply $\epsilon_i$ by $\pi$ and send all other standard basis elements to themselves.  Then there is a unique isomorphism of chains of $\O_F$-modules
\begin{equation}\label{disp:latt_triv}
   \vcenter{
   \xymatrix{
      \dotsb\, \ar@{^{(}->}[r]
         & \Lambda_0 \ar@{^{(}->}[r] 
         & \Lambda_1 \ar@{^{(}->}[r]
         & \,\dotsb\, \ar@{^{(}->}[r] 
         & \Lambda_{n} \ar@{^{(}->}[r]
         & \,\dotsb\\
      \dotsb\, \ar[r]^-{\beta_{n}}
         & \O_F^{n} \ar[u]_\sim \ar[r]^-{\beta_1}
         & \O_F^{n} \ar[u]_-\sim \ar[r]^-{\beta_2}
         & \,\dotsb\, \ar[r]^-{\beta_{n}}
         & \O_F^{n} \ar[u]_\sim \ar[r]^-{\beta_1}
         & \,\dotsb
   }
   }
\end{equation}
such that the leftmost displayed vertical arrow identifies the ordered $\O_F$-basis $\epsilon_1,\dotsc,\epsilon_n$ of $\O_F^{n}$ with the ordered basis $e_1,\dotsc,e_{n}$ of $\Lambda_0$.  Restricting to subchains in the top and bottom rows in \eqref{disp:latt_triv}, we get an analogous trivialization of $\Lambda_I$ for any $I$.

\subsection{Naive local models}\label{ss:naive_lm}
We now review the definition of the naive local models from \cite{paprap09}*{\s1.5}.

Recall our fixed partition $n = s + r$ with $s < r$.  The \emph{naive local model
$M_I^\naive$} is the following contravariant functor on the
category of $\O_F$-algebras.  Given an $\O_F$-algebra $R$, an $R$-point in $M_I^\naive$ consists of, up to an obvious notion of isomorphism,
\begin{itemize}
\item
   a functor
   \[
      \xymatrix@R=0ex{
         \Lambda_I \ar[r] & {}(\text{$\O_F \tensor_{\O_{F_0}} R$-modules})\\
         \Lambda_i \ar@{|->}[r] & \F_i,
      }
   \]
   where $\Lambda_I$ is regarded as a category by taking the morphisms to be the inclusions of lattices in $V$; together with
\item
   an inclusion of $\O_F \tensor_{\O_{F_0}} R$-modules $\F_i \inj \Lambda_i \tensor_{\O_{F_0}} R$ for each $i \in n\ZZ \pm I$, functorial in $\Lambda_i$;
\end{itemize}
satisfying the following conditions for all $i \in n\ZZ \pm I$.
\begin{enumerate}
\renewcommand{\theenumi}{LM\arabic{enumi}}
\item 
   Zariski-locally on $\Spec R$, $\F_i$ embeds in $\Lambda_i \tensor_{\O_{F_0}} R$ as a direct $R$-module summand of rank $n$.
\item\label{it:period_cond}
   The isomorphism $\Lambda_i \tensor_{\O_{F_0}} R \isoarrow \Lambda_{i - n} \tensor_{\O_{F_0}} R$ obtained by tensoring $\Lambda_i \xra[\sim]\pi \pi \Lambda_i = \Lambda_{i - n}$ identifies $\F_i$ with $\F_{i - n}$.
\item\label{it:perp_cond}
   The perfect $R$-bilinear pairing 
   \[
      (\Lambda_i \tensor_{\O_{F_0}} R) 
         \times (\Lambda_{-i} \tensor_{\O_{F_0}} R)
      \xra{\text{$\langle$~,~$\rangle$} \tensor R}
      R
   \]
   induced by \eqref{disp:perf_pairing} identifies $\F_i^\perp \subset \Lambda_{-i} \tensor_{\O_{F_0}} R$ with $\F_{-i}$.
\item\label{it:kottwitz_cond}
   The element $\pi \tensor 1 \in \O_F \tensor_{\O_{F_0}} R$ acts on the $\O_F \tensor_{\O_{F_0}} R$-module $\F_i$ as an $R$-linear endomorphism with characteristic polynomial
   \[
      \det(T\cdot \id - \pi \tensor 1 \mid \F_i) = (T - \pi)^s (T+ \pi)^{r} \in R[T].
   \]
\end{enumerate}

The functor $M_I^\naive$ is plainly represented by a closed subscheme, which
we again denote $M_I^\naive$, of a finite product of Grassmannians over
$\Spec \O_F$.  The generic fiber of $M_I^\naive$ can be identified with the Grassmannian of $s$-planes in an $n$-dimensional vector space; see \cite{paprap09}*{\s1.5.3}.


\subsection{Wedge condition}\label{ss:wedge_cond}
As observed by Pappas \cite{pap00}, the naive local model $M_I^\naive$ often fails to be flat over $\O_F$.  As a first step towards correcting for non-flatness, Pappas proposed the addition of a new condition, the \emph{wedge condition}, to the moduli problem defining $M_I^\naive$:  this is the condition that for a given $R$-point $(\F_i)_i$ of $M_I^\naive$,
\begin{enumerate}
\setcounter{enumi}{4}
\renewcommand{\theenumi}{LM\arabic{enumi}}
\item
   for all $i \in n\ZZ \pm I$,
   \[
      \sideset{}{_R^{s+1}}{\bigwedge} (\pi\tensor 1 + 1 \tensor \pi \mid \F_i) = 0
      \quad\text{and}\quad
      \sideset{}{_R^{r+1}}{\bigwedge} (\pi\tensor 1 - 1 \tensor \pi \mid \F_i) = 0.
   \]
\end{enumerate}
We denote by $M_I^\wedge$ the subfunctor of $M_I^\naive$ of points that satisfy the wedge condition, and we call it the \emph{wedge local model}.  Plainly $M_I^\wedge$ is a closed subscheme of $M_I^\naive$.  As noted in \cite{paprap09}*{\s1.5.6}, the generic fibers of $M_I^\wedge$ and $M_I^\naive$ coincide.

\subsection{Spin condition}\label{ss:spin_cond}
Although the wedge local model turns out to be flat in some cases in which the naive local model is not, Pappas and Rapoport have observed that it also fails to be flat in general.  As a further --- and, conjecturally, last --- step towards correcting for non-flatness, in \cite{paprap09}*{\s7} they proposed the addition of a further condition to the moduli problem, the \emph{spin condition}.  In this subsection we review their formulation of the spin condition; compare also with \cite{sm09a}*{\s2.3}.  For sake of brevity, we shall recall only the bare minimum of linear algebra we need.

Regarding $V$ as a $2n$-dimensional vector space over $F_0$, consider the ordered $F_0$-basis
\[
   -\pi^{-1}e_1, \dotsc, -\pi^{-1}e_m, e_{m+1}, \dotsc, e_{n}, e_1,\dotsc, e_{m}, \pi e_{m+1}, \dotsc, \pi e_{n},
\]
which we denote by $f_1',\dotsc$, $f_{2n}'$.  Extending scalars to $F$, we get the ordered $F$-basis $f_1' \tensor 1, \dotsc$, $f_{2n}' \tensor 1$ for $V \tensor_{F_0} F$.  We then define a new basis $f_1,\dotsc,$ $f_{2n}$ for $V \tensor_{F_0} F$ by taking
\[
   f_i := f_i'\tensor 1
   \quad \text{for} \quad
   i \neq m+1,\ n+m+1
\]
and by replacing $f_{m+1}' \tensor 1 = e_{m+1} \tensor 1$ and $f_{n+m+1}' \tensor 1 = \pi e_{m+1} \tensor 1$ with
\[
   f_{m+1} := e_{m+1} \tensor 1 - \pi e_{m+1} \tensor \pi^{-1}
   \quad\text{and}\quad
   f_{n+m+1} := \frac{e_{m+1} \tensor 1 + \pi e_{m+1} \tensor \pi^{-1}}2.
\]
Next recall from \eqref{disp:pairings} the $F_0$-bilinear symmetric form $($~,~$)$ on $V$, and let us continue to write $($~,~$)$ for its base change to $V \tensor_{F_0} F$.  Then the basis $f_1,\dotsc,$ $f_{2n}$ is \emph{split} for  $($~,~$)$, that is, we have $(f_i,f_j) = \delta_{i,2n+1-j}$ for all $i$ and $j$.

We now use the split ordered basis $f_1,\dotsc,$ $f_{2n}$ to define an operator $a$ on $\bigwedge_F^{n} (V \tensor_{F_0} F)$.  For a subset $E \subset \{1,\dotsc,2n\}$ of cardinality $n$, let
\begin{equation}\label{disp:f_E}
   f_E := f_{i_1}\wedge\dotsb\wedge f_{i_{n}} \in \sideset{}{_F^{n}}{\bigwedge} (V \tensor_{F_0} F),
\end{equation}
where $E = \{i_1,\dotsc,i_{n}\}$ with $i_1 < \dotsb < i_{n}$.  Given such $E$, we also let 
\[
   E^\perp := (2n+1-E)^c = 2n+1-E^c,
\]
where the set complements are taken in $\{1,\dotsc,2n\}$.  Then $E^\perp$ consists of the elements $i'\in\{1,\dotsc,2n\}$ such that $(f_i,f_{i'}) = 0$ for all $i\in E$.  We now define $a$ by defining it on the basis elements $f_E$ of $\bigwedge_F^{n} (V \tensor_{F_0} F)$ for varying $E$:
\[
   a(f_E) := \sgn(\sigma_E)f_{E^\perp},
\]
where $\sigma_E$ is the permutation on $\{1,\dotsc,2n\}$ sending $\{1,\dotsc,n\}$ to the elements of $E$ in increasing order, and sending $\{n+1,\dotsc,2n\}$ to the elements of $E^c$ in increasing order.

\begin{rk}\label{rk:a_sign_difference}
Our operator $a$ agrees \emph{only up to sign} with the analogous operators denoted $a_{f_1\wedge \dotsb \wedge f_{2n}}$ in \cite{paprap09}*{display 7.6} and $a$ in \cite{sm09a}*{\s2.3}.  Indeed, these latter operators send
\[
   f_E \mapsto \sgn(\sigma'_E) f_{E^\perp},
\]
where $\sigma_E'$ is the permutation on $\{1,\dotsc,2n\}$ sending $\{1,\dotsc,n\}$ to the elements of $2n+1-E$ in decreasing order, and sending $\{n+1,\dotsc,2n\}$ to the elements of $E^\perp$ in increasing order.  We have
\[
   \sigma_E' = \rho \circ \sigma_E \circ \tau,
\]
where $\rho$ sends $i \mapsto 2n+1-i$, and $\tau$ fixes $\{1,\dotsc,n\}$ and sends $\{n+1,\dotsc,2n\}$ to itself in decreasing order.  Hence
\[
   \sgn(\sigma_E') = (-1)^n \cdot \sgn(\sigma_E) \cdot (-1)^m = (-1)^{m+1}\sgn(\sigma_E).
\]
\end{rk}

Returning to the main discussion, it follows easily from the definition of $\sigma_E$, or directly from \cite{paprap09}*{7.1} and the preceding remark, that $\sgn(\sigma_E) = \sgn(\sigma_{E^\perp})$.  Hence $a^2 = \id_{\bigwedge^n (V \tensor_{F_0} F)}$.  Hence $\bigwedge_F^n (V \tensor_{F_0} F)$ decomposes as
\[
   \sideset{}{_F^n}{\bigwedge} (V \tensor_{F_0} F) 
      = \Bigl(\sideset{}{_F^n}{\bigwedge} (V \tensor_{F_0} F)\Bigr)_{1} 
            \oplus \Bigl(\sideset{}{_F^n}{\bigwedge} (V \tensor_{F_0} F)\Bigr)_{-1},
\]
where
\[
   \Bigl(\sideset{}{_F^n}{\bigwedge} (V \tensor_{F_0} F)\Bigr)_{\pm 1} 
     := \spn_F\bigl\{f_E \pm \sgn(\sigma_E)f_{E^\perp}\bigr\}_E
\]
is the $\pm 1$-eigenspace for $a$; here $E$ ranges through the subsets of $\{1,\dotsc,2n\}$ of cardinality $n$ in the last display.

Now consider the $\O_F$-lattice $\Lambda_i \subset V$ for some $i$.
Then $\Lambda_i \tensor_{\O_{F_0}} \O_F$ is naturally an $\O_F$-lattice in $V \tensor_{F_0} F$, and $\bigwedge^n_{\O_F} (\Lambda_i \tensor_{\O_{F_0}} \O_F)$ is naturally an $\O_F$-lattice in $\bigwedge^n_F (V \tensor_{F_0} F)$.  We set
\[
   \Bigl(\sideset{}{_{\O_F}^n}{\bigwedge} (\Lambda_i \tensor_{\O_{F_0}} \O_F)\Bigr)_{\pm 1}
      := \Bigl(\sideset{}{_{\O_F}^n}{\bigwedge} (\Lambda_i \tensor_{\O_{F_0}} \O_F) \Bigr)
            \cap \Bigl(\sideset{}{_F^n}{\bigwedge} (V \tensor_{F_0} F)\Bigr)_{\pm 1}.
\]

We are finally ready to state the spin condition.  Recall our partition $n = s+r$, and note that, given an $R$-point $(\F_i)_i$ of $M_I^\naive$, the $R$-module $\bigwedge_{R}^n \F_i$ is naturally contained in $\bigwedge^n_{R} (\Lambda_i \tensor_{\O_{F_0}} R)$ for all $i$.  The \emph{spin condition} is that
\begin{enumerate}
\renewcommand{\theenumi}{LM\arabic{enumi}}
\setcounter{enumi}{5}
\item
   for all $i \in n\ZZ \pm I$, $\bigwedge_{R}^n \F_i$ is contained in
   \[
      \im\biggl[
         \Bigl(\sideset{}{_{\O_F}^n}{\bigwedge} 
            (\Lambda_i \tensor_{\O_{F_0}} \O_F)\Bigr)_{(-1)^s}\tensor_{\O_F} R 
         \to \sideset{}{_{R}^n}{\bigwedge} (\Lambda_i \tensor_{\O_{F_0}} R)
      \biggr].
   \]
\end{enumerate}
For a fixed index $i$, we say that $\F_i$ satisfies the spin condition if $\bigwedge_R^n \F_i$ is contained in the displayed image.
We denote by $M_I^\spin$ the subfunctor of $M_I^\wedge$ of points satisfying the spin condition, and we call it the \emph{spin local model}.  Plainly $M_I^\spin$ is a closed subscheme of $M_I^\wedge$.  As noted in \cite{paprap09}*{\s7.2.2}, the generic fiber of $M_I^\spin$ agrees with the common generic fiber of $M_I^\naive$ and $M_I^\wedge$.

\begin{rk}
The statement of the spin condition in \cite{paprap09}*{\s7.2.1} actually contains a sign error which traces to the sign discrepancy observed in \eqref{rk:a_sign_difference}.  Indeed, the element $f_1\wedge\dotsb\wedge f_n$ lies in the $(-1)^{m+1}$-eigenspace for the operator $a_{f_1\wedge\dotsb\wedge f_{2n}}$ of \cite{paprap09}.  Thus the argument of \cite{paprap09}*{\s7.2.2} shows that the sign of $(-1)^s$ in the statement of the spin condition in \cite{paprap09} should be replaced with $(-1)^{s+m+1}$.  For us, since $f_1\wedge\dotsb\wedge f_n$ lies in the $+1$-eigenspace of our operator $a$, the same argument shows that we get a sign of $(-1)^s$.
\end{rk}

\subsection{Lattice chain automorphisms}\label{ss:latt_chain_auts}
Regarding $\Lambda_I$ as a lattice chain over $\O_{F_0}$, the perfect pairings $\langle$~,~$\rangle$ of \eqref{disp:perf_pairing} give a \emph{polarization} of $\Lambda_I$ in the sense of \cite{rapzink96}*{3.14} (with $B = F$, ${b}^* = \ol{b}$ in the notation of \cite{rapzink96}).  Consider the $\O_{F_0}$-group scheme $\ul\Aut(\Lambda_I)$, the scheme of automorphisms of the lattice chain $\Lambda_I$ that preserve the pairings $\langle$~,~$\rangle$ for variable $\Lambda_i \in \Lambda_I$ up to common unit scalar.  Then $\ul\Aut(\Lambda_I)$ is smooth and affine over $\O_{F_0}$; see \cite{pap00}*{2.2}, which in turn relies on \cite{rapzink96}*{3.16}.  Let \A denote the base change of $\ul\Aut(\Lambda_I)$ to $\O_F$.  Then \A acts naturally on $M_I^\naive$, and it is easy to see that this action preserves the closed subschemes $M_I^\wedge$, $M_I^\spin$, and $M_I^\loc$, where we recall that $M_I^\loc$ denotes the scheme-theoretic closure of the generic fiber in $M_I^\naive$.  We shall return to this point in \s\ref{ss:embedding}.

\section{Unitary similitude group}\label{s:GU}
In this section we review a number of basic group-theoretic matters from \cite{paprap09}; these will become relevant in the next section when we begin to consider the affine flag variety.  We switch to working with respect to the auxiliary field extension $K/K_0$.  We write $i^* := n+1-i$ for $i \in \{1,\dotsc,n\}$.

\subsection{Unitary similitudes}
Let $h$ denote the Hermitian form on $K^{n}$ whose matrix with respect to the standard ordered basis is \eqref{disp:antidiag_1}.  We denote by $G := GU_{n} := GU(h)$ the algebraic group over $K_0$ of \emph{unitary similitudes} of $h$: for any $K_0$-algebra $R$, $G(R)$ is the group of elements $g\in GL_{n}(K \tensor_{K_0} R)$ satisfying $h_R(gx,gy) = c(g)h_R(x,y)$ for some $c(g) \in R^\times$ and all $x$, $y\in (K \tensor_{K_0} R)^{n}$, where $h_R$ is the induced form on $(K \tensor_{K_0} R)^{n}$.  As the form $h$
is nonzero, the scalar $c(g)$ is uniquely determined, and $c$ defines an
exact sequence of $K_0$-groups
\[
   1 \to U_n \to G \xra c \GG_m \to 1
\]
with evident kernel $U_n:= U(h)$ the unitary group of $h$.

After base change to $K$, we get the standard identification
\begin{equation}\label{disp:G_K=GL_ntimesG_m}
   G_K \xra[\sim]{(\varphi,c)} GL_{n,K} \times \GG_{m,K},
\end{equation}
where $\varphi\colon G_K \to GL_{n,K}$ is the map $x\tensor y \mapsto xy$ on matrix entries.

\subsection{Tori}\label{ss:tori}
We denote by $S$ the standard diagonal maximal split torus in $G$:  on $R$-points,
\[
   S(R) = \bigl\{\, \diag(a_1,\dotsc,a_{n}) \in GL_{n}(R) \bigm| 
          a_1 a_{n} = \dotsb = a_m a_{m+2} = a_{m+1}^2 \,\bigr\}.
\]
The centralizer $T$ of $S$ is the standard maximal torus of all diagonal matrices in $G$,
\[
   T(R) = \biggl\{\, \diag(a_1,\dotsc,a_n) \in GL_n(K\tensor_{K_0} R) \biggm|
      \begin{varwidth}{\textwidth}
         \centering
         $a_1 \ol a_{n} = \dotsb = a_m \ol a_{m+2}$\\ 
         $= a_{m+1} \ol a_{m+1}$
      \end{varwidth} \,\biggr\}.
\]
The isomorphism \eqref{disp:G_K=GL_ntimesG_m} identifies $T_K$ with the split torus $D \times \GG_{m,K}$, where $D$ denotes the standard diagonal maximal torus in $GL_{n,K}$.  The standard identification $X_*(D \times \GG_{m,K}) \ciso \ZZ^n \times \ZZ$ then identifies the inclusion $X_*(S) \subset X_*(T)$ with
\begin{equation}\label{disp:X_*S_in_X_*T}
   \bigl\{\, (x_1,\dotsc,x_n,y) \bigm| x_1 + x_n = \dotsb = x_m + x_{m+2} = 2x_{m+1} = y \,\bigr\}
   \subset \ZZ^n \times \ZZ;
\end{equation}
note that this is not the description of $X_*(S)$ given in \cite{paprap09}*{\s2.4.2}, which appears to contain an error.

For later use, it is convenient to introduce here the cocharacter $\mu_{r,s} \in X_*(T)$ given in terms of our above identifications as
\begin{equation}\label{disp:mu_s}
   \mu_{r,s} := \bigl(1^{(s)},0^{(r)},1\bigr) \in \ZZ^n \times \ZZ.
\end{equation}

We write \aaa for the standard apartment $X_*(S) \tensor_\ZZ \RR$, and we regard it as a sub\-space of $\RR^n \times \RR$ via \eqref{disp:X_*S_in_X_*T}.

\subsection{Affine roots}
In terms of the identification \eqref{disp:X_*S_in_X_*T}, the \emph{relative roots} of $S$ in $G$ are shown in \cite{paprap09}*{\s2.4.2} (modulo the description of $X_*(S)$ there) to consist of the maps $\{\pm\alpha_{i,j}\}_{i < j < i^*} \cup \{\pm\alpha_{i,i^*}\}_{i < m+1}$ on $X_*(S)$, where
\[
   \alpha_{i,j} \colon (x_1,\dotsc,x_n,y) \mapsto
   x_i - x_j.
\]
The \emph{affine roots} are then shown to consist of the maps
\[
   \pm\alpha_{i,j} + \tfrac 1 2 \ZZ
   \quad\text{for}\quad
   i < j <i^*
   \quad\text{and}\quad
   \pm \alpha_{i,i^*} + \tfrac 1 2 + \ZZ
   \quad\text{for}\quad
   i < m+1.
\]
Thus the affine root hyperplanes consist of the zero loci of the affine functions
\begin{equation}\label{disp:C_m_roots_unscaled}
   \begin{gathered}
      \pm2\alpha_{i,j} + \ZZ
      \quad\text{for}\quad
      i < j <i^*,\ j \neq m+1;\quad\text{and}\\
      \pm 2\alpha_{i,i^*} + \ZZ
      \quad\text{for}\quad
      i < m+1.
   \end{gathered}
\end{equation}
These last may be regarded as the affine roots attached to a root system of type $C_m$.

\subsection{Iwahori-Weyl group}
The \emph{Iwahori-Weyl group} of $G$ with respect to the maximal split torus $S$ is the group
\[
   \wt W_G := N(K_0)/T(K_0)_1,
\]
where $N$ is the normalizer of $T$ in $G$ and $T(K_0)_1$ is the kernel of the Kottwitz homomorphism $T(K_0) \surj X_*(T)_{\Gal(\ol K_0 / K_0)} = X_*(T)_\Gamma$.  The evident exact sequence
\[
   1 \to T(K_0)/T(K_0)_1 \to \wt W_G \to N(K_0)/T(K_0) \to 1
\]
splits, so that $\wt W_G$ is expressible as a semidirect product
\[
   \wt W_G \iso X_*(T)_\Gamma \rtimes W_G,
\]
where $W_G:= N(K_0)/T(K_0)$ is the relative Weyl group of $S$ in $G$.

Concretely, the permutation matrices in $G(K_0)$ map isomorphically onto $W_G$, and in this way we identify $W_G$ with $S_n^*$.  On the other hand, the nontrivial element in $\Gamma$ acts on $X_*(T)$ by sending
\[
   (x_1,\dotsc,x_n,y) \mapsto (y-x_n,\dotsc,y-x_1,y).
\]
Hence the surjective map
\[
   \xymatrix@R=0ex{
      X_*(T) \ar@{->>}[r] & \ZZ^m\times \ZZ\\
      (x_1,\dotsc,x_n,y) \ar@{|->}[r] & (x_1 - x_n, x_2 - x_{n-1},\dotsc,x_m - x_{m+2},y)
   }
\]
identifies the coinvariants $X_*(T)_\Gamma$ with $\ZZ^m\times\ZZ$.  Moreover, it is now clear from our various identifications that the composite $X_*(S) \to X_*(T) \to X_*(T)_\Gamma$ identifies $X_*(S)$ with $2X_*(T)_\Gamma$.  Hence we may just as well identify $X_*(T)_\Gamma$ with $\frac 1 2 X_*(S)$ inside $X_*(T)\tensor\QQ \ciso \QQ^n \times \QQ$.  Of course, in this way $W_G$ acts on $X_*(T)_\Gamma$ via its action on all of $\QQ^n \times \QQ$, with the natural permutation action of $S_n^*$ on the first $n$ factors and the trivial action on the last factor.

To better facilitate working directly with $X_*(T)_\Gamma$, it is convenient to now change coordinates on \aaa.  Starting from the coordinates defined in \s\ref{ss:tori}, let us multiply by $2$ and project $(x_1,\dotsc,x_n,y) \mapsto (x_1,\dotsc,x_n)$, so that we now identify
\begin{equation}\label{disp:X_*(T)_Gamma_identification}
   X_*(T)_\Gamma \iso \bigl\{\, (x_1,\dotsc,x_{n}) \in \ZZ^{n} \bigm| x_1 + x_{n} = \dotsb = x_m + x_{m+2} = 2 x_{m+1} \,\bigr\},
\end{equation}
and we replace \eqref{disp:X_*S_in_X_*T} with the identification
\begin{equation}\label{disp:X_*(S)_identification}
   X_*(S) \iso \bigl\{\, (x_1,\dotsc,x_{n}) \in 2\ZZ^{n} \bigm| x_1 + x_n = \dotsb = x_m + x_{m+2} = 2x_{m+1} \,\bigr\}.
\end{equation}
Under the identification \eqref{disp:X_*(T)_Gamma_identification}, the Kottwitz map $T(K_0) \surj X_*(T)_\Gamma$ has the simple form
\[
   \diag(a_1,\dotsc,a_n) \mapsto (\ord_u a_1,\dotsc,\ord_u a_n).
\]
With respect to our new coordinates, the affine functions \eqref{disp:C_m_roots_unscaled} correspond to the functions
\begin{equation}\label{disp:C_m_roots}
   \begin{gathered}
      \pm\alpha_{i,j} + \ZZ
      \quad\text{for}\quad
      i < j <i^*,\ j \neq m+1;\quad\text{and}\\
      \pm \alpha_{i,i^*} + \ZZ
      \quad\text{for}\quad
      i < m+1.
   \end{gathered}
\end{equation}

Tracing through our various identifications, we also note that with respect to our new coordinates, the image in $X_*(T)_\Gamma$ of the cocharacter $\mu_{r,s}$ \eqref{disp:mu_s} identifies with
\begin{equation}\label{disp:ol_mu_s}
   \bigl(2^{(s)},1^{(n-2s)},0^{(s)}\bigr) \in \ZZ^n.
\end{equation}

\subsection{Bruhat order}\label{ss:bo}
We now briefly review the Bruhat order on $\wt W_G$.  The reflections in the apartment \aaa across the affine root hyperplanes --- or what are the same, across the zero loci of the functions \eqref{disp:C_m_roots} --- are naturally elements in $\wt W_G$; the \emph{affine Weyl group $W_{\aff,G}$} is the subgroup of $\wt W_G$ generated by them.  The affine Weyl group acts simply transitively on the set of alcoves in \aaa, and the choice of a base alcove $A$ presents $\wt W_G$ as a semidirect product $W_{\aff,G} \rtimes \Omega_A$, where $\Omega_A$ is the stabilizer of $A$ in $\wt W_G$.  The reflections across the walls of $A$ generate $W_{\aff,G}$ as a Coxeter group, so that $W_{\aff,G}$ is endowed with a Bruhat order $\leq$.  The Bruhat order then extends to $\wt W_G$ in the usual way: for $x\omega$, $x'\omega' \in \wt W_G$ with $x$, $x'\in W_{\aff,G}$ and $\omega$, $\omega'\in\Omega_A$, we have $x\omega \leq x'\omega'$ exactly when $\omega = \omega'$ and $x\leq x'$ in $W_{\aff,G}$.

\subsection{Relation to the symplectic group}\label{ss:GSp_2m}
Let
\[
   X_* := \bigl\{\, (x_1,\dotsc,x_{2m}) \in \ZZ^{2m} \bigm| x_1 + x_{2m} = \dotsb = x_m + x_{m+1}\,\bigr\},
\]
and consider the Iwahori-Weyl group $\wt W_{GSp_{2m}} := X_* \rtimes S_{2m}^*$ of the split symplectic similitude group with respect to its diagonal maximal torus.  The identification \eqref{disp:X_*(T)_Gamma_identification} makes plain that the map on cocharacters
\[
   (x_1,\dotsc,x_n) \mapsto (x_1,\dotsc,x_m,x_{m+2},\dotsc,x_n)
\]
induces an embedding $\wt W_G \inj \wt W_{GSp_{2m}}$ as a subgroup of index $2$.  Moreover, \eqref{disp:C_m_roots} makes plain that in this way, the affine root hyperplane structure on \aaa identifies with the affine root hyperplane structure in the standard apartment for $GSp_{2m}$.  In particular, the Bruhat order on $\wt W_G$ is inherited from the Bruhat order on $\wt W_{GSp_{2m}}$; we have  
\[
   W_{\aff,G} = Q^\vee_G \rtimes W_G \subset X_*(T)_\Gamma \rtimes W_G,
\]
where $Q^\vee_G$ identifies with
\[
   \bigl\{\, (x_1,\dotsc,x_n) \in \ZZ^n \bigm| x_1 + x_{n} = \dotsb = x_m + x_{m+2} = x_{m+1} = 0 \,\bigr\};
\]
and the stabilizer group $\Omega_A$ (for any alcove $A$) maps isomorphically to the quotient
\[
   \wt W_G/W_{\aff,G} \ciso X_*(T)_\Gamma / Q^\vee_G \iso \ZZ.
\]

\subsection{Parahoric subgroups}\label{ss:parahoric}
We next recall the description of the parahoric subgroups of $G(K_0)$ from \cite{paprap09}.  In analogy with \s\ref{ss:lattices}, for $i = nb+c$ with $0 \leq c < n$, we define the $\O_{K}$-lattice
\[
   \lambda_i := \sum_{j=1}^c u^{-b-1}\O_K e_j + \sum_{j=c+1}^{n} u^{-b}\O_K e_j \subset K^n,
\]
where now $e_1,\dotsc,e_n$ denotes the standard ordered basis in $K^n$.  For any nonempty subset $I \subset \{0,\dotsc,m\}$, we write $\lambda_I$ for the chain consisting of all lattices $\lambda_i$ for $i \in n\ZZ \pm I$, and we define
\begin{align*}
   P_I &:= \bigl\{\, g\in G(K_0) \bigm| g\lambda_i = \lambda_i \text{ for all } i \in n\ZZ \pm I \,\bigr\}\\
       &\phantom{:}= \bigl\{\, g\in G(K_0) \bigm| g\lambda_i = \lambda_i \text{ for all } i \in I \,\bigr\}.
\end{align*}
We have the following.

\begin{prop}[\cite{paprap09}*{\s1.2.3(a)}]\label{st:parahoric_description}
$P_I$ is a parahoric subgroup of $G(K_0)$, and every parahoric subgroup of $G(K_0)$ is conjugate to $P_I$ for a unique nonempty $I \subset \{0,\dotsc,m\}$.  The sets $I = \{0\}$ and $I = \{m\}$ correspond to the special maximal parahoric subgroups.\qed
\end{prop}

\subsection{Base alcove}\label{ss:base_alc}
From now on, we take as our base alcove $A$ the unique alcove fixed by the Iwahori subgroup $P_{\{0,\dotsc,m\}}$.  In terms of our coordinates \eqref{disp:X_*(S)_identification} on $X_*(S)$, we have
\[
   A = \biggl\{ \,(x_1,\dotsc,x_n) \in \RR^n \biggm|
            \begin{varwidth}{\textwidth}
               \centering
                $x_1 + x_n = \dotsb = x_m + x_{m+2} = 2x_{m+1}$ and\\
                $x_n-1 < x_1 < x_2 < \dotsb < x_m < x_{m+2}$ 
            \end{varwidth}
            \,\biggr\}
            \subset \aaa.
\]
Of course, the choice of $A$ determines a Bruhat order on $\wt W_G$, as discussed in \s\ref{ss:bo}.

\subsection{Coset and double coset variants}\label{ss:coset_variants}
To consider local models for general parahoric, not just Iwahori, level structure, it is necessary to consider certain double coset variants of $\wt W_G$, which we now review from the paper of Kottwitz and Rapoport \cite{kottrap00}*{\s8}.  For nonempty $I \subset \{0,\dotsc,m\}$, let
\[
   W_{G,I} := \bigl( N(K_0) \cap P_I \bigr) \big/ T(K_0)_1
   \subset W_{\aff,G} \subset \wt W_G.
\]
For nonemtpy $I$, $J \subset \{0,\dotsc,m\}$, consider the set of double cosets $W_{G,I} \bs \wt W_G / W_{G,J}$; this inherits a Bruhat order from $\wt W_G$ in the following way.  Let $W_{G,I} \wt w W_{G,J}$ and $W_{G,I}\wt u W_{G,J}$ be double cosets, and let $\wt w_0$, $\wt u_0 \in \wt W_G$ be their respective unique elements of minimal length; see Bourbaki \cite{bourLGLA4-6}*{IV \s1 Ex.\ 3}.  Then $W_{G,I}\wt w W_{G,J} \leq W_{G,I}\wt u W_{G,J}$ in the Bruhat order exactly when $\wt w_0 \leq \wt u_0$ in $\wt W_G$.

\begin{rk}\label{rk:bo_double_cosets}
We recall from \cite{kottrap00}*{8.3} the following facts about the Bruhat order on $W_{G,I} \bs \wt W_G / W_{G,J}$: if $\wt w \leq \wt u$ in $\wt W_G$, then $W_{G,I}\wt w W_{G,J} \leq W_{G,I}\wt u W_{G,J}$; and if $W_{G,I}\wt w W_{G,J} \leq W_{G,I}\wt u W_{G,J}$ and $\wt w_0$ is the element of minimal length in $W_{G,I}\wt w W_{G,J}$, then $\wt w_0 \leq \wt u$.
\end{rk}

\section{Affine flag variety}\label{s:afv}
From now on we take $K_0 = k((t))$, $\O_{K_0} = k[[t]]$, $K = k((u))$, and $\O_K = k[[u]]$.  In this section we review the affine flag variety attached to  a parahoric subgroup of $G$, and the embedding of the special fiber of the naive local model into it, from \cite{paprap09}*{\s3}.  Let $I \subset \{0,\dotsc,m\}$ be a nonempty subset.

\subsection{Affine flag variety}
Let $P$ be a parahoric subgroup of $G(K_0)$.  Then Bruhat-Tits theory attaches to $P$ a smooth $\O_{K_0}$-group scheme whose generic fiber is identified with $G$, whose special fiber is connected, and whose group of $\O_{K_0}$-points is identified with $P$; abusing notation, we denote this group scheme again by $P$.  The \emph{affine flag variety $\F_P$ relative to $P$} is the fpqc quotient of functors on the category of $k$-algebras,
\[
   \F_P := LG/L^+P,
\]
where $LG$ is the \emph{loop group} $LG\colon R \mapsto G\bigl(R((t))\bigr)$ and $L^+P$ is the \emph{positive loop group}  $L^+P\colon R \mapsto G\bigl(R[[t]]\bigr)$.  See \cite{paprap08}.  The affine flag variety is an ind-$k$-scheme of ind-finite type.

\subsection{Lattice-theoretic description}\label{ss:lattice_theoretic}
In this subsection we give a slight variant of (and make a minor correction to) the description in \cite{paprap09}*{\s3.2} of the affine flag variety in terms of lattice chains in $K^n$.  

Let $R$ be a $k$-algebra.  Recall that an \emph{$R[[u]]$-lattice in $R((u))^{n}$} is an $R[[u]]$-submodule $L \subset R((u))^{n}$ which is free as an $R[[u]]$-module Zariski-locally on $\Spec R$, and such that the natural arrow $L \tensor_{R[[u]]} R((u)) \to R((u))^{n}$ is an isomorphism.  Borrowing notation from \eqref{disp:wh_Lambda}, given an $R[[u]]$-lattice $L$, we write $\wh L$ for the dual lattice
\[
   \wh L := \bigl\{\, x \in R((u))^{n} \bigm| h_{R((u))}(L,x) \subset R[[u]] \, \bigr\},
\]
where $h_{R((u))} := h \tensor_K R((u))$ is the induced form on $R((u))^{n}$. A collection $\{L_i\}_i$ of $R[[u]]$-lattices in $R((u))^{n}$ is a \emph{chain} if it is totally ordered under inclusion and all successive quotients are locally free $R$-modules (necessarily of finite rank).  A lattice chain is \emph{periodic} if $u L$ is in the chain for every lattice $L$ in the chain.  We write $\L\bigl(R((u))^n\bigr)$ for the category whose objects are the $R[[u]]$-lattices in $R((u))^n$ and whose morphisms are the natural inclusions of lattices.  Of course, any $R[[u]]$-lattice chain may be regarded as a full subcategory of $\L\bigl(R((u))^n\bigr)$.

We define $\F_I$ to be the functor on $k$-algebras that assigns to each $R$ the set of all functors $L\colon\lambda_I \to \L\bigl(R((u))^n\bigr)$ such that
\begin{enumerate}
\renewcommand{\theenumi}{C}
\item\label{it:afv_chain_cond}
   (chain) the image $L(\lambda_I)$ is a lattice chain in $R((u))^n$;
\renewcommand{\theenumi}{P}
\item\label{it:afv_periodicity_cond}
   (periodicity) $L(u\lambda_i) = uL(\lambda_i)$ for all $i \in n\ZZ \pm I$, so that the chain $L(\lambda_I)$ is periodic;
\renewcommand{\theenumi}{R}
\item\label{it:afv_rank_cond}
   (rank) $\dim_k \lambda_i/\lambda_j = \rank_R L(\lambda_i) / L(\lambda_j)$ for all $j  < i$; and
\renewcommand{\theenumi}{D}
\item\label{it:afv_dlty_cond}
   (duality) Zariski-locally on $\Spec R$, there exists $\alpha \in R((t))^\times \subset R((u))^\times$ such that $\wh{L(\lambda_i)} = \alpha L\bigl(\wh\lambda_i\bigr)$ for all $i \in n\ZZ \pm I$.
\end{enumerate}

If $L \in \F_I(R)$ globally admits a scalar $\alpha$ as in \eqref{it:afv_dlty_cond}, then $\alpha$ is well-defined modulo the group
\[
   \bigl\{\, a \in R((t))^\times \bigm| a L(\lambda_{-i}) = L(\lambda_{-i}) \,\bigr\}
   = R[[t]]^\times
\]
(independent of $i$).  It is then an easy exercise to check that the map specified locally by
\[
   L \mapsto \Bigl(\bigl( L(\lambda_i) \bigr)_{i \in I},\ \alpha \bmod R[[t]]^\times \Bigr)
\]
defines an isomorphism from the functor $\F_I$ as we've defined it to the functor $\F_I$ as defined in \cite{paprap09}*{\s3.2}, except that the functor in \cite{paprap09} should only require that $\alpha \bmod R[[t]]^\times$ be given Zariski-locally.  The loop group $LG$ acts on $\F_I$ via the natural representation of $G\bigl(R((t))\bigr)$ on $R((u))^n$, and it follows that the $LG$-equivariant map $LG \to \F_I$ specified by taking the tautological inclusion $\bigl(\lambda_I \inj \L(K^n)\bigr) \in \F_I(k)$ as basepoint defines an $LG$-equivariant isomorphism
\[
   \F_{P_I} \isoarrow \F_I.
\]
We shall always identify $\F_{P_I}$ and $\F_I$ in this way.

\subsection{Schubert cells and varieties}
Consider the parahoric subgroup scheme $P_I$ over $\O_{K_0}$ and its associated affine flag variety $\F_{P_I}$.  For $n \in N(K_0)$, the associated \emph{Schubert cell} is the reduced $k$-subscheme
\[
   L^+P_I \cdot n \subset \F_{P_I}.
\]
The Schubert cell depends only on the image $w$ in $W_{G,I} \bs \wt W_G / W_{G,I}$ of $n$, and we denote the Schubert cell by $C_w$.  By Haines and Rapoport \cite{hrap08}*{Prop.~8}, the inclusion $N(K_0) \subset G(K_0)$ induces a bijection
\[
   W_{G,I} \bs \wt W_G / W_{G,I} 
   \isoarrow P_I(\O_{K_0}) \bs G(K_0) / P_I(\O_{K_0}),
\]
so that the Schubert cells are indexed by precisely the elements of $W_{G,I} \bs \wt W_G / W_{G,I}$ and give a stratification of all of $\F_{P_I}$.  Note that in the special case $I = \{0,\dotsc,m\}$, $P_I$ is an Iwahori subgroup, the group $W_{G,I}$ is trivial, and the Schubert cells are indexed by $\wt W_G$ itself.

For $w\in W_{G,I} \bs \wt W_G / W_{G,I}$, the associated \emph{Schubert variety $S_w$} is the reduced closure of $C_w$ in $\F_{P_I}$.  The closure relations between Schubert cells are given by the Bruhat order:  for $w$, $w' \in W_{G,I} \bs \wt W_G / W_{G,I}$, we have $S_w \subset S_{w'}$ in $\F_{P_I}$ $\iff$ $w \leq w'$ in $W_{G,I} \bs \wt W_G / W_{G,I}$.

\subsection{Embedding the special fiber}\label{ss:embedding}
We conclude the section by recalling from \cite{paprap09}*{\s3.3} the embedding of the special fiber $M_{I,k}^\naive := M_I^\naive \tensor_{\O_K} k$ of the naive local model in the affine flag variety $\F_{P_I} \iso \F_I$.

The embedding makes use of the lattice-theoretic description of the affine flag variety from \s\ref{ss:lattice_theoretic}.  First note that the $\O_K$-lattice chain $\lambda_I$ admits a trivialization in obvious analogy with the trivialization of $\Lambda_I$ specified by \eqref{disp:latt_triv}, where $\lambda_i$ replaces $\Lambda_i$, $\O_K$ replaces $\O_F$, and $u$ replaces $\pi$.  Upon identifying $\O_K \tensor_{\O_{K_0}} k \isoarrow \O_F \tensor_{\O_{F_0}} k$ by sending the $k$-basis elements $1 \tensor 1 \mapsto 1 \tensor 1$ and $u \tensor 1 \mapsto \pi \tensor 1$, our trivializations then yield an identification of chains of $k$-vector spaces
\begin{equation}\label{disp:chain_isom}
   \lambda_i \tensor_{\O_{K_0}} k \iso \Lambda_i \tensor_{\O_{F_0}} k;
\end{equation}
this is even an isomorphism of $k[u]/(u^2)$-modules, where $u$ acts on the right-hand side as multiplication by $\pi \tensor 1$.

Now let $R$ be a $k$-algebra.  Given an $R$-point $\{\F_i\}_i$ in $M_I^\naive$, for each $i$, let $L_i \subset \lambda_i \tensor_{\O_{K_0}} R[[t]]$ denote the inverse image of
\[
   \F_i \subset \Lambda_i \tensor_{\O_{F_0}} R \iso \lambda_i \tensor_{\O_{K_0}} R
\]
under the reduction-mod-$t$-map
\[
   \lambda_i \tensor_{\O_{K_0}} R[[t]]
   \surj \lambda_i \tensor_{\O_{K_0}} R.
\]
Then $L_i$ is naturally a lattice in $R((u))^n$, and the functor $\lambda_I \to \L\bigl( R((u))^n \bigr)$ sending $\lambda_i \mapsto L_i$ determines a point in $\F_I(R)$.  In this way we get a monomorphism
\begin{equation}\label{disp:embedding}
   \vcenter{
   \xymatrix@R=0ex{
      M_{I,k}^\naive\,  \ar@{^{(}->}[r]  &  \F_I\\
      \{\F_i\}_i  \ar@{|->}[r]  & (\lambda_i \shortmapsto L_i)
   }
   }.
\end{equation}
Since $M_{I,k}^\naive$ is proper, \eqref{disp:embedding} is a closed immersion of ind-schemes.  From now on we shall often identify $M_{I,k}^\naive$ with its image in $\F_I$.

The embedding \eqref{disp:embedding} is \emph{$L^+P_I$-equivariant} with respect to the following left $L^+P_I$-actions on source and target; compare \citelist{\cite{paprap03}*{\s3}\cite{paprap05}*{\s6, \s11}\cite{paprap08}*{\s11}\cite{paprap09}*{\s3.3}}.  For $\F_I$ we just take the natural action furnished by our isomorphism $\F_I \iso \F_{P_I}$.  For $M_{I,k}^\naive$, the tautological action of $P_I$ on $\lambda_I$ yields a natural action of $L^+P_I$ on $\lambda_I \tensor_{\O_{K_0}} k$. The chain isomorphism \eqref{disp:chain_isom} then induces a homomorphism $L^+P_I \to \A \tensor_{\O_{F_0}} k$, where we recall the $\O_{F}$-group scheme \A from \s\ref{ss:latt_chain_auts}.  The \A-action on $M_I^\naive$ now furnishes the desired $L^+P_I$-action on $M_{I,k}^\naive$.  Of course, in this way $L^+P_I$ also acts on $M_{I,k}^\wedge$, $M^\spin_{I,k}$, and $M_{I,k}^\loc$.

\section{Schubert cells in the special fiber}\label{s:schub_cells}
We continue to take $I$ to be a nonempty subset of $\{0,\dotsc,m\}$.  In this section we describe the Schubert cells that are contained in the images of $M_{I,k}^\wedge$ and $M_{I,k}^\spin$ in $\F_I$, and we reduce the Main Theorem to showing that these Schubert cells are indexed by the set of $\mu_{r,s}$-admissible elements in $W_{G,I}\bs\wt W_G/W_{G,I}$, where $\mu_{r,s}$ is the cocharacter \eqref{disp:mu_s}.  We continue to write $i^* := n+1-i$.

\subsection{The image of the special fiber}\label{ss:im_spec_fib}
Let $R$ be a $k$-algebra.  It is clear from the definition of the embedding $M_{I,k}^\naive \inj \F_I$ \eqref{disp:embedding} and the various conditions in the definition of $M_I^\naive$ that the image of $M^\naive_{I,k}(R)$ in $\F_I(R)$ consists precisely of the functors $\lambda_i \mapsto L_i$ in $\F_I(R)$ such that, for all $i \in n\ZZ \pm I$,
\begin{enumerate}
\item\label{it:latt_incl_cond}
   $\lambda_{i,R[[t]]} \supset L_i \supset t\lambda_{i,R[[t]]}$, where $\lambda_{i,R[[t]]} := \lambda_i \tensor_{\O_{K_0}} R[[t]]$;
\item\label{it:rank_n_cond}
   the $R$-module $\lambda_{i,R[[t]]}/L_i$ is locally free of rank $n$; and
\item\label{it:t^-1_cond}
   $\wh L_i = t^{-1} L_{-i}$.
\end{enumerate}

\begin{rk}
Condition \eqref{it:t^-1_cond} is actually redundant; that is, for any functor $\lambda_I \to \L \bigl(R((u))^n \bigr)$ satisfying conditions \eqref{it:afv_chain_cond}, \eqref{it:afv_periodicity_cond}, \eqref{it:afv_rank_cond}, and \eqref{it:afv_dlty_cond} from \s\ref{ss:lattice_theoretic} and conditions \eqref{it:latt_incl_cond} and \eqref{it:rank_n_cond} above, the scalar $\alpha$ appearing in \eqref{it:afv_dlty_cond} must be congruent to $t^{-1} \bmod R[[t]]^\times$.  We leave the details as an exercise to the reader.
\end{rk}

Returning to the main discussion, it is an immediate consequence of $L^+P_I$-equivariance of \eqref{disp:embedding} that \emph{the underlying topological spaces of $M^\naive_{I,k}$, $M_{I,k}^\wedge$, $M^\spin_{I,k}$, and $M_{I,k}^\loc$ are all unions of Schubert varieties in $\F_I$.}  One of our essential goals for the rest of the paper is to obtain a good description of the Schubert varieties that occur in $M_{I,k}^\wedge$ and $M^\spin_{I,k}$.

\subsection{Faces of type \texorpdfstring{$I$}{\emph{I}}}\label{ss:faces_type_I}
As a first step towards our goal, we recall the combinatorial notion of a \emph{face of type $I$} of Kottwitz and Rapoport \cite{kottrap00}*{\s\s9--10}; see also \cite{goertz05}*{\s7}.%
\footnote{Strictly speaking, we shall define what \cites{kottrap00,goertz05} would call a face of type $n\ZZ \pm I$, but we shall ignore this difference.}
Similarly to \cite{sm09a}*{\s5.6}, we shall adopt some different conventions (corresponding to a different choice of base alcove) to better facilitate working with the affine flag variety.

Given an integer $d$, a \emph{$d$-face of type $I$} is a family $(v_i)_{i\in n\ZZ\pm I}$ of vectors $v_i \in \ZZ^n$ such that
\begin{enumerate}
\renewcommand{\theenumi}{F\arabic{enumi}}
\item\label{it:face_per_cond}
   $v_{i+n} = v_i - \mathbf 1$ for all $i \in n\ZZ \pm I$;
\item
   $v_i \geq v_j$ for all $i$, $j \in n\ZZ \pm I$ with $i \leq j$;
\item\label{it:face_adjcy_cond}
   $\Sigma v_i - \Sigma v_j = j-i$ for all $i$, $j \in n\ZZ \pm I$; and
\item\label{it:face_dlty_cond}
   $v_i + v_{-i}^* = \mathbf d$ for all $i \in n\ZZ \pm I$.
\end{enumerate}
A family of vectors $(v_i)_{i\in n\ZZ \pm I}$ is a \emph{face of type $I$} if it is a $d$-face of type $I$ for some $d$.  Since $n$ is odd, it is an easy consequence of \eqref{it:face_adjcy_cond} and \eqref{it:face_dlty_cond} that $d$-faces only occur for even $d$.

For $i = nb + c$ with $b \in \ZZ$ and $0 \leq c < n$, we define
\begin{equation}\label{disp:omega_i}
   \omega_i := \bigl((-1)^{(c)},0^{(n-c)}\bigr) - \mathbf b.
\end{equation}
The family $\omega_I := (\omega_i)_{i\in n\ZZ \pm I}$ is a $0$-face of type $I$, which we call the \emph{standard face of type $I$}.  The natural action of $\wt W_G = X_*(T)_\Gamma \rtimes S_n^*$, with $X_*(T)_\Gamma$ embedded in $\ZZ^n$ as in \eqref{disp:X_*(T)_Gamma_identification}, by affine transformations on $\ZZ^n$ induces a transitive action of $\wt W_G$ on faces of type $I$.  The stabilizer of $\omega_I$ in $\wt W_G$ is plainly $W_{G,I}$, and we identify the faces of type $I$ with $\wt W_G / W_{G,I}$ by taking $\omega_I$ as basepoint.  Note that in the Iwahori case $I = \{0,\dotsc,m\}$, the action of $\wt W_G$ on faces of type $I$ is simply transitive, so that these are identified with $\wt W_G$ itself.

\subsection{The vector \texorpdfstring{$\mu_i$}{mu\_i}}
Given a face $(v_i)_i$ of type $I$, let
\[
   \mu_i := v_i - \omega_i, \quad i \in n\ZZ \pm I.
\]
Then condition \eqref{it:face_per_cond} is equivalent to the periodicity relation
\begin{equation}\label{disp:mu_per_cond}
   \mu_i = \mu_{i+n} \quad \text{for all}\quad i,
\end{equation}
and \eqref{it:face_adjcy_cond} implies
\begin{equation}\label{disp:mu_adjcy_cond}
   \Sigma\mu_i = \Sigma\mu_j \quad \text{for all}\quad i,\ j.
\end{equation}
If $(v_i)_i$ is a $d$-face, then condition \eqref{it:face_dlty_cond} is equivalent to
\begin{equation}\label{disp:mu_dlty_cond}
   \mu_i + \mu_{-i}^* = \mathbf d \quad \text{for all}\quad i.
\end{equation}

We now prove a couple of lemmas for later use.  For $i \in I$, let
\[
   A_i := \{1,2,\dotsc,i,i^*,i^*+1,\dotsc,n\}
   \quad\text{and}\quad
   B_i := \{i+1,i+2,\dotsc,n-i\}.
\]

\begin{lem}[Basic inequalities]\label{st:basic_ineqs}
Suppose $(v_i)_i$ is a $d$-face of type $I$.  Then for any $i \in I$, we have
\[
   d \leq \mu_i(j) + \mu_i(j^*) \leq d+1 \quad\text{for}\quad j \in A_i
\]
and
\[
   d-1 \leq \mu_i(j) + \mu_i(j^*) \leq d \quad\text{for}\quad j \in B_i.
\]
\end{lem}

\begin{proof}
It suffices to prove the lemma in the Iwahori case, that is, when $I = \{0,\dotsc,m\}$.  Modulo conventions related to the choice of base alcove, this is done for $d = 0$ in \cite{sm09b}*{4.4.1}, and the argument there works just as well for arbitrary $d$.
\end{proof}

\begin{rk}\label{rk:basic_ineqs_arb_index}
For $i \in I$, it follows immediately from the lemma and \eqref{disp:mu_dlty_cond} that
\[
   d-1 \leq \mu_{-i}(j) + \mu_{-i}(j^*) \leq d \quad\text{for}\quad j \in A_i
\]
and
\[
   d \leq \mu_{-i}(j) + \mu_{-i}(j^*) \leq d+1 \quad\text{for}\quad j \in B_i.
\]
The periodicity relation \eqref{disp:mu_per_cond} on the $\mu$'s now gives analogous basic inequalities on the entries of $\mu_{i'}$ for any $i' \in n\ZZ \pm I$.  In particular, since $d$ must be even, we conclude that $\mu_{i'}(m+1) = d/2$ for all $i'$.
\end{rk}

We turn to our second lemma.  Let $i \in n\ZZ \pm I$.  We say that $\mu_i$ is \emph{self-dual} if $\mu_i = \mu_{-i}$, or in other words, if $\mu_i + \mu_i^* = \mathbf d$ where $(v_i)_i$ is a $d$-face.

\begin{lem}\label{st:self-dual_mu}
Suppose that $(v_i)_i$ is a $d$-face of type $I$.  Let $i \in I$, and for $\mu \in \{\mu_i,\mu_{-i}\}$, suppose that the equality $\mu(j) + \mu(j^*) = d$ holds for all $j \in A_i$ or for all $j \in B_i$.  Then $\mu$ is self-dual.\end{lem}

\begin{proof}
The relations \eqref{disp:mu_adjcy_cond} and \eqref{disp:mu_dlty_cond} imply $\Sigma \mu = nd/2$, and the conclusion is then an obvious consequence of the basic inequalities.
\end{proof}

Of course, the statement of the lemma can be extended in an obvious way to $\mu_{i'}$ for any $i' \in n\ZZ\pm I$ by using periodicity.

\subsection{Naive permissibility}\label{ss:naive_perm}
Let $w \in W_{G,I} \bs \wt W_G / W_{G,I}$.  Since the orbit $\wt W_G \cdot \lambda_I$ meets every Schubert cell in $\F_I$, the condition that a cell $C_w$ be contained in $M_{I,k}^\naive$ (resp., $M_{I,k}^\wedge$; resp., $M_{I,k}^\spin$) amounts to the condition that the point $\wt w \cdot \lambda_I$ be contained in $M_{I,k}^\naive$ (resp., $M_{I,k}^\wedge$; resp., $M_{I,k}^\spin$), where $\wt w$ is any representative of $w$ in $\wt W_G/W_{G,I}$.  We shall find it convenient to express these containment conditions in terms of faces of type $I$, beginning in this subsection with containment in $M_{I,k}^\naive$.

Let $(v_i)_i := \wt w \cdot \omega_I$ denote the face of type $I$ attached to $\wt w$.  Then it is clear from the definitions and from \s\ref{ss:im_spec_fib} that $C_w$ is contained in $M_{I,k}^\naive$ $\iff$
\begin{enumerate}
\renewcommand{\theenumi}{P\arabic{enumi}}
\item\label{it:ineq_cond}
   $\omega_i \leq v_i \leq \omega_i + \mathbf 2$ for all $i \in n\ZZ \pm I$; and
\item\label{it:size_cond}
   $\Sigma v_i = n - i$ for one, hence every, $i \in n\ZZ \pm I$.
\end{enumerate}
We say that such a $\wt w$ is \emph{naively permissible}.  If $\wt w$ is  naively permissible, then necessarily $(v_i)_i$ is a $2$-face.

Given a naively permissible $\wt w$, the point $\wt w\cdot \lambda_I$ in $\F_I(k)$ corresponds to a point $(\F_i \subset \Lambda_i \tensor_{\O_{F_0}} k)_i$ in $M_{I,k}^\naive(k)$ of a rather special sort:  namely, identifying $\Lambda_i \tensor_{\O_{F_0}} k$ with $\O_F^n \tensor_{\O_{F_0}} k$ via \eqref{disp:latt_triv}, we have
\begin{enumerate}
\renewcommand{\theenumi}{S}
\item\label{disp:S_fixed}
   for all $i$, $\F_i$, regarded as a subspace in $\O_F^n \tensor_{\O_{F_0}} k$, is $k$-spanned by $n$ of the elements $\epsilon_1\tensor 1,\dotsc$, $\epsilon_n\tensor 1$, $\pi\epsilon_1 \tensor 1,\dotsc$, $\pi\epsilon_n \tensor 1$,
\end{enumerate}
where we recall from \s\ref{ss:lattices} that $\epsilon_1,\dotsc,\epsilon_n$ denotes the standard basis in $\O_F^n$.  On the other hand, for any point $(\F_i)_i$ in $M_{I,k}^\naive(k)$, let us say that $(\F_i)_i$ is an \emph{$S$-fixed point} if it satisfies \eqref{disp:S_fixed}; it is easy to check that the $S$-fixed points are exactly the points in $M_k^\naive(k)$ fixed by $L^+S(k)$.  In this way, we get a bijection between the naively permissible $\wt w\in \wt W_G/W_{G,I}$ and the $S$-fixed points in $M_{I,k}^\naive(k)$, which we denote by $\wt w \mapsto (\F_i^{\wt w})_i$.

The $S$-fixed point $(\F^{\wt w}_i)_i$ attached to a  naively permissible $\wt w$ is conveniently described in terms of the face $(v_i)_i$ of type $I$ attached to $\wt w$.  Indeed, let
\[
   \mu^{\wt w}_i := v_i - \omega_i, \quad i \in n\ZZ \pm I.
\]
Then $\Sigma\mu^{\wt w}_i = n$,
\begin{equation}\label{disp:0=<mu=<2}
   \mathbf 0 \leq \mu^{\wt w}_i \leq \mathbf 2, 
\end{equation}
and
\begin{equation}\label{disp:F_i}
   \F_i^{\wt w} = \sum_{\mu^{\wt w}_i(j) = 0} k\cdot (\epsilon_j\tensor 1)
                          +\sum_{\mu^{\wt w}_i(j) = 0,1} k \cdot (\pi\epsilon_j \tensor 1) \subset \O_F^n \tensor_{\O_{F_0}} k.
\end{equation}

\subsection{Wedge- and spin-permissibility}\label{ss:wedge_spin_perm}
Let $\wt w \in \wt W_G/W_{G,I}$ be naively permissible, let $(v_i)_i$ denote its associated face of type $I$, let $\mu^{\wt w}_i := v_i - \omega_i$ for all $i$, and let $(\F^{\wt w}_i)_i \in M_{I,k}^\naive(k)$ denote the associated $S$-fixed point.  We say that $\wt w$ is \emph{wedge-permissible} (resp.\ \emph{spin-permiss\-i\-ble}) if the Schubert cell in $\F_I$ attached to $\wt w$ is contained in $M_{I,k}^\wedge$ (resp., $M_{I,k}^\spin$).  Our aim in this subsection is to express the conditions of wedge- and spin-permissibility in terms of the $v_i$'s and $\mu^{\wt w}_i$'s.

We begin with wedge-permissibility. By definition,
\[
   (\F^{\wt w}_i)_i \in M_{I,k}^\wedge(k) \iff
   \begin{varwidth}{\textwidth}
      \centering
      for all $i$, $\bigwedge_k^{s+1} (\pi\tensor 1 \mid \F^{\wt w}_i) = 0$\\
      and $\bigwedge_k^{r+1} (\pi\tensor 1 \mid \F^{\wt w}_i) = 0$,
   \end{varwidth}
\]
where we recall our fixed partition $n = s + r$ with $s < r$.  For fixed $i$, the second equality on the right-hand side of the display is implied by the first.  Hence we read off the following from \eqref{disp:F_i}.

\begin{prop}\label{st:wedge-perm}
$\wt w \in \wt W_G/W_{G,I}$ is wedge-permissible $\iff$ $\wt w$ is naively permissible and
\begin{enumerate}
\renewcommand{\theenumi}{P3}
\item
   for all $i \in n\ZZ \pm I$, $\#\smash{\bigl\{\, j \bigm| \mu^{\wt w}_i(j) = 0\, \bigr\}} \leq s$.\qed
\end{enumerate}
\end{prop}

We next turn to spin-permissibility.  We are going to show that for our naively permissible $\wt w$, \emph{the point $(\F^{\wt w}_i)_i\in M_{I,k}^\naive(k)$ already satisfies the spin condition,} regardless of the parity of $s$.  Our discussion will largely parallel \cite{sm09a}*{\s7.5}.  Of course, the spin condition is a condition on $\F^{\wt w}_i$ that must be checked for each $i \in I \cup (-I)$.  We shall only do so explicitly for $i\in I$, leaving the entirely analogous case $i \in -I$ to the reader.

Fix $i\in I$.  We continue to identify $\O_F^n$ with $\Lambda_i$ via \eqref{disp:latt_triv}.  Since $\mu_i^{\wt w}(m+1)$ necessarily equals $1$ \eqref{rk:basic_ineqs_arb_index}, the $n$ elements in $\O_F^n \tensor_{\O_{F_0}} \O_F$
\begin{equation}\label{disp:F_i_spanners}
\begin{gathered}
   \epsilon_j \tensor 1 \quad\text{for}\quad \mu^{\wt w}_i(j) = 0;\\
   \pi\epsilon_j \tensor 1 \quad\text{for}\quad \mu^{\wt w}_i(j) = 0,1, \ j \neq m+1;\ \text{and}\\
   \pi\epsilon_{m+1} \tensor 1 \pm \epsilon_{m+1} \tensor \pi
\end{gathered}
\end{equation}
span an $\O_F$-submodule whose image under the reduction map $\O_F^n \tensor_{\O_{F_0}} \O_F \surj \O_F^n \tensor_{\O_{F_0}} k$ is $\F^{\wt w}_i$.  (For now we shall allow ourselves the choice of either sign in the last element in \eqref{disp:F_i_spanners}.)  Take the wedge product (in any order) of the $n$ elements \eqref{disp:F_i_spanners} in $\bigwedge_{\O_F}^n (\O_F^n \tensor_{\O_{F_0}} \O_F)$, and let $f \in \bigwedge_{\O_F}^n (\Lambda_i \tensor_{\O_{F_0}} \O_F)$ denote the image of this element under the isomorphism $\bigwedge_{\O_F}^n (\O_F^n \tensor_{\O_{F_0}} \O_F) \isoarrow \bigwedge_{\O_F}^n (\Lambda_i \tensor_{\O_{F_0}} \O_F)$ induced by \eqref{disp:latt_triv}.  Then, up to a sign and factor of $2$, and in terms of the notation \eqref{disp:f_E}, $f$ equals
\[
   \pi^q f_{E_\pm} \in \sideset{}{_F^n}{\bigwedge} (V \tensor_{F_0} F),
\]
where $E_\pm \subset \{1,\dotsc,2n\}$ is the subset of cardinality $n$
{\allowdisplaybreaks
\begin{align*}
   E_\pm := {}&\bigl\{\, j \in \{1,\dotsc,i\} \bigm| \mu^{\wt w}_i(j) = 0 \,\bigr\}\\
                   &\quad\amalg \bigl\{\, j \in \{i+1,\dotsc,m\} \bigm| \mu^{\wt w}_i(j) = 0,1 \,\bigr\}\\
                   &\quad\amalg \bigl\{\, j \in \{m+2,\dotsc,n\} \bigm| \mu^{\wt w}_i(j) = 0 \,\bigr\}\\
                   &\quad\amalg\bigl\{\, n+j \in \{n+1,\dotsc,n+i\} \bigm| \mu^{\wt w}_i(j) = 0,1 \,\bigr\}\\
                   &\quad\amalg\bigl\{\, n+j \in \{n+i+1,\dotsc,n+m\} \bigm| \mu^{\wt w}_i(j) = 0 \,\bigr\}\\
                   &\quad\amalg\bigl\{\, n+j \in \{n+m+2,\dotsc,2n\} \bigm| \mu^{\wt w}_i(j) = 0,1 \,\bigr\}\\
                   &\quad\amalg \{b_\pm\},
\end{align*}
}%
where $b_- := m+1$ and $b_+ := n+m+1$, and we choose the sign according to the choice of sign in \eqref{disp:F_i_spanners}; and where
\[
   q := 1 + 2\cdot \#\bigl(E_\pm \cap \{i+1,\dotsc,m\}\bigr) = 1 + 2\cdot \#\bigl\{\, j \in \{i+1,\dotsc,m\} \bigm| \mu^{\wt w}_i(j) = 0,1 \,\bigr\}.
\]

To study the spin condition for $\F^{\wt w}_i$, we also need the set $E_\pm^\perp = (2n+1-E_\pm)^c$, which is given by
{\allowdisplaybreaks
\begin{align*}
   E_\pm^\perp = {}&\bigl\{\, j \in \{1,\dotsc,m\} \bigm| \mu^{\wt w}_i(j^*) \neq 0,1 \,\bigr\}\\
          &\quad\amalg \bigl\{\, j \in \{m+2,\dotsc,n-i\} \bigm| \mu^{\wt w}_i(j^*) \neq 0 \,\bigr\}\\
          &\quad\amalg \bigl\{\, j \in \{i^*,\dotsc,n\} \bigm| \mu^{\wt w}_i(j^*) \neq 0,1 \,\bigr\}\\
          &\quad\amalg\bigl\{\, n+j \in \{n+1,\dotsc,n+m\} \bigm| \mu^{\wt w}_i(j^*) \neq 0 \,\bigr\}\\
          &\quad\amalg\bigl\{\, n+j \in \{n+m+2,\dotsc,2n-i\} \bigm| \mu^{\wt w}_i(j^*) \neq 0,1 \,\bigr\}\\
          &\quad\amalg\bigl\{\, n+j \in \{n+i^*,\dotsc,2n\} \bigm| \mu^{\wt w}_i(j^*) \neq 0 \,\bigr\}\\
          &\quad\amalg \{b_\pm\}.
\end{align*}
}%
Up to a sign and factor of $2$, the element in $\bigwedge_{\O_F}^n (\Lambda_i \tensor_{\O_{F_0}} \O_F)$
\[
\begin{split}
    &(\pi e_{m+1} \pm e_{m+1}\tensor \pi)
    \wedge
    \bigwedge_{j \in E_\pm^\perp \cap \{1,\dotsc,i\}} (\pi^{-1} e_j \tensor 1)\\
   &\qquad\wedge
   \bigwedge_{j \in E_\pm^\perp \cap \{i+1,\dotsc,m\}} (\pi e_j \tensor 1)
   \qquad\wedge
   \bigwedge_{j \in E_\pm^\perp \cap \{m+2,\dotsc,n\}} (e_j \tensor 1)\\
   &\qquad\qquad\wedge
   \bigwedge_{n+j \in E_\pm^\perp \cap \{n+1,\dotsc,n+m\}} (e_j \tensor 1)
   \qquad\wedge
   \bigwedge_{n+j \in E_\pm^\perp \cap \{n+m+2,\dotsc,2n\}} (\pi e_j \tensor 1)
\end{split}
\]
equals
\[
   \pi^{q^\perp} f_{E_\pm^\perp} \in \sideset{}{_F^n}{\bigwedge} (V \tensor_{F_0} F),
\]
where
\[
   q^\perp := 1 + 2 \cdot \#\bigl(E_\pm^\perp \cap \{i+1,\dotsc,m\}\bigr) = 1 + 2 \cdot \#\bigl\{\, j \in \{i+1,\dotsc,m\} \bigm| \mu^{\wt w}_i(j^*) = 2 \,\bigr\}.
\]

Comparing our expressions for $q$ and $q^\perp$, we deduce immediately from \eqref{disp:0=<mu=<2} and the basic inequalities \eqref{st:basic_ineqs} that $q \geq q^\perp$.  We shall now consider separately the cases $q > q^\perp$ and $q = q^\perp$.

If $q > q^\perp$, then let $E := E_\pm$, and consider the elements
\begin{equation}\label{disp:disp_elts}
   \pi^q f_{E} \pm \pi^{q - q^\perp} \pi^{q^\perp} \sgn(\sigma_E) f_{E^\perp}
   \in\Bigl(\sideset{}{_{\O_F}^n}{\bigwedge} (\Lambda_i \tensor_{\O_{F_0}} \O_F)\Bigr)_{\pm 1},
\end{equation}
where we allow either sign in \eqref{disp:disp_elts} independently of the sign in \eqref{disp:F_i_spanners}, and where the notation is as in \s\ref{ss:spin_cond}.  By definition of $\pi^q f_E$, for either choice of sign in \eqref{disp:F_i_spanners}, the common image of the elements \eqref{disp:disp_elts} in $\bigwedge_{k}^n (\Lambda_i \tensor_{\O_{F_0}} k)$ spans the line $\bigwedge_k^n \F^{\wt w}_i$.  Hence $\F^{\wt w}_i$ satisfies the spin condition.

If $q = q^\perp$, then we deduce at once from the basic inequalities that for all pairs $j$, $j^* \in \{i+1,\dotsc,m,m+2,\dotsc,n-i\}$, one of the entries $\mu^{\wt w}_i(j)$, $\mu^{\wt w}_i(j^*)$ is $0$ and the other is $2$; and, as always, $\mu^{\wt w}_i(m+1) = 1$.  Thus $\mu^{\wt w}_i$ is self-dual by \eqref{st:self-dual_mu}, and we read off from the explicit expressions for $E_\pm$ and $E_\pm^\perp$ that $E_\pm = E_\pm^\perp$.  It is clear from the definitions that $\sgn(\sigma_{E_+}) = - \sgn(\sigma_{E_-})$, and therefore one of the two elements $\pi^q f_{E_+}$, $\pi^q f_{E_-}$ is contained in $\bigl(\bigwedge_{\O_F}^n (\Lambda_i \tensor_{\O_{F_0}} \O_F)\bigr)_{+1}$ and the other is contained in $\bigl(\bigwedge_{\O_F}^n (\Lambda_i \tensor_{\O_{F_0}} \O_F)\bigr)_{-1}$.  Since the common image in $\bigwedge_k^n (\Lambda_i \tensor_{\O_{F_0}}k)$ of these elements spans $\bigwedge_k^n \F^{\wt w}_i$, we conclude that $\F^{\wt w}_i$ satisfies the spin condition.

%

We have now shown that the point $(\F^{\wt w}_i)_i \in M_{I,k}^\naive(k)$ satisfies the spin condition, and with it the following.

\begin{prop}\label{st:wedge_iff_spin-perm}
$\wt w \in \wt W_G/W_{G,I}$ is spin-permissible $\iff$ $\wt w$ is wedge-permissi\-ble.\qed
\end{prop}

\begin{cor}\label{st:toplogical_agreement}
The schemes $M_I^\wedge$ and $M_I^\spin$ have the same reduced underlying subschemes.\qed
\end{cor}

\subsection{Topological flatness of \texorpdfstring{$M_I^\wedge$}{M\_I\^{}wedge} and \texorpdfstring{$M_I^\spin$}{M\_I\^{}spin}}\label{ss:top_flat}

We now come to the main results of the paper.  The key combinatorial fact we shall need in the proof of topological flatness for $M_I^\wedge$ and $M_I^\spin$ is the equivalence of wedge- and spin-permissibility with  \emph{$\mu_{r,s}$-admissibility}, where $\mu_{r,s}$ is the cocharacter \eqref{disp:mu_s}.  Recall that for any cocharacter $\mu$ of $T$, the element $\wt w \in \wt W_G/W_{G,I}$ is \emph{$\mu$-admissible} if there exists  $\sigma \in W_G$ such that $\wt w \leq t_{\sigma\cdot \mu}W_{G,I}$ in the Bruhat order, where $t_\mu$ denotes the image of $\mu$ in $X_*(T)_\Gamma$ regarded as an element in $\wt W_G$.

\begin{thm}\label{st:adm<=>perm_I}
Let $\wt w \in \wt W_G/W_{G,I}$.  Then $\wt w$ is wedge-permissible $\iff$ $\wt w$ is spin-permissible $\iff$ $\wt w$ is $\mu_{r,s}$-admissible.
\end{thm}

\begin{proof}
We have already seen the equivalence of wedge-permissibility and spin-permissibility in \eqref{st:wedge_iff_spin-perm}.  We shall show that wedge-permissibility is equivalent to $\mu_{r,s}$-admissibility in \s\ref{s:combinatorics}.
\end{proof}

The notions appearing in the theorem all make sense in an obvious way for double cosets as well: $w \in  W_{G,I}\bs\wt W_G/W_{G,I}$ is respectively \emph{wedge-} or \emph{spin-permissible} if $\wt w$ is wedge- or spin-permissible for one, hence any, representative $\wt w \in \wt W_G/W_{G,I}$; and $w$ is \emph{$\mu$-admissible} if there exists  $\sigma \in W_G$ such that $w \leq W_{G,I}t_{\sigma\cdot \mu}W_{G,I}$ in the Bruhat order.

\begin{cor}\label{st:adm<=>perm_II}
Let  $w \in W_{G,I}\bs\wt W_G/W_{G,I}$.  Then $w$ is wedge-permissible $\iff$ $w$ is spin-permissible $\iff$ $w$ is $\mu_{r,s}$-admissible.
\end{cor}

\begin{proof}
The first $\Longleftrightarrow$ is immediate from \eqref{st:adm<=>perm_I}, and the second follows from this and \eqref{rk:bo_double_cosets}.
\end{proof}

\begin{cor}\label{st:top_flat}
The schemes $M_I^\wedge$ and $M_I^\spin$ are topologically flat over $\O_F$.
\end{cor}

\begin{proof}
Since the set of $\mu_{r,s}$-admissible elements in $\wt W_G$ surjects onto the set of $\mu_{r,s}$-admissible elements in $W_{G,I}\bs\wt W_G/W_{G,I}$, \cite{paprap09}*{3.1} shows exactly that the Schubert cells $C_w$ indexed by $\mu_{r,s}$-admissible $w$ are contained in $M_{I,k}^\loc$ (taking note that, set-theoretically, $M_{I,k}^\loc$ is itself a union of Schubert cells).  We now get exactly what we need from \eqref{st:adm<=>perm_II}.
\end{proof}

\section{Combinatorics}\label{s:combinatorics}

In this final section we prove the equivalence of wedge-permissibility with $\mu_{r,s}$-admissibility needed in \eqref{st:adm<=>perm_I}, as well as the equivalence of these notions with Kottwitz and Rapoport's notion of $\mu_{r,s}$-permissibility (see \s\ref{ss:perm}).  As before, we fix a nonempty subset $I \subset \{0,\dotsc,m\}$.  To lighten notation, we set $\wt W := \wt W_{GSp_{2m}} = X_* \rtimes S_{2m}^*$.

\subsection{Formulation in terms of \texorpdfstring{$GSp_{2m}$}{GSp\_2m}}\label{ss:transfer_to_GSp_2m}
In this subsection we use the embedding $\wt W_G \inj \wt W$ from \s\ref{ss:GSp_2m} to transfer the problem of proving the equivalence between wedge-permissibility and $\mu_{r,s}$-admissibility to an equivalent problem for $GSp_{2m}$.

Changing notation from \eqref{disp:omega_i}, we now denote by $\omega_i$ the vector
\[
   \omega_i := \bigl((-1)^{(c)},0^{(2m-c)}\bigr) - \mathbf b
\]
for $i = 2mb + c$ with $b \in \ZZ$ and $0 \leq c < 2m$.  We write $W_I$ for the stabilizer in $\wt W$ of all the vectors $\omega_i$ with $i \in 2m\ZZ \pm I$; this is the image in $\wt W$ of $W_{G,I}$.  We say that $\wt w \in \wt W /W_I$ is \emph{wedge-permissible} if for all $i \in 2m\ZZ \pm I$,
\begin{enumerate}
\renewcommand{\theenumi}{P$'$\arabic{enumi}}
\item\label{it:naive_perm_1}
   $\mathbf 0 \leq \wt w \cdot \omega_i - \omega_i \leq \mathbf 2$;
\item\label{it:naive_perm_2}
   $\Sigma(\wt w \cdot \omega_i - \omega_i) = 2m$; and
\item\label{it:wedge_perm}
   $\#\{\, j\mid (\wt w \cdot \omega_i - \omega_i)(j) = 0\, \} \leq s$.
\end{enumerate}
Trivially, the wedge-permissible elements in $\wt W$ are just the images of the wedge-permissible elements in $\wt W_G$ under the embedding $\wt W_G \inj \wt W$.  It is now clear from \eqref{disp:ol_mu_s} and from the discussion in \s\ref{ss:GSp_2m} that our problem is to show the following.

\begin{thm}\label{st:wedge_equiv_adm}
Let  $\wt w \in \wt W/W_I$.  Then $\wt w$ is wedge-permissible in $\wt W/W_I$ $\iff$ $\wt w$ is $\mu$-admissible in $\wt W/W_I$ for $\mu$ the coweight $\bigl(2^{(s)},1^{(2m-2s)},0^{(s)}\bigr) \in \ZZ^{2m}$.
\end{thm}
 
Of course, here the Bruhat order on $\wt W$ is taken with respect to the alcove
\begin{equation}\label{disp:GSp_alc}
   \biggl\{\, (x_1,\dotsc,x_{2m}) \in \RR^{2m} \biggm|
      \begin{varwidth}{\textwidth}
         \centering
         $x_1 + x_{2m} = \dotsb = x_m + x_{m+1}$ and\\
         $x_{2m}-1 < x_1 < x_2 < \dotsb < x_m < x_{m+1}$
      \end{varwidth}
      \,\biggr\},
\end{equation}
or in other words, the alcove contained in the Weyl chamber opposite the standard positive chamber and whose closure contains the origin.  We shall complete the proof of the theorem in \s\ref{ss:wedge-perm_mu-adm}.

\subsection{A lemma on Steinberg fixed-point root data}\label{ss:steinberg_fprd}
The key input we shall use to establish \eqref{st:wedge_equiv_adm} is a theorem of Haines and Ng\^o which describes admissible sets for $GSp_{2m}$.  Strictly speaking, their theorem applies only to the Iwahori case, and the aim of this subsection is to prove a general lemma which will aid us in extending their result to the general parahoric case.

We shall formulate our lemma in the setting of general Steinberg fixed-point root data.  Changing notation, in this subsection (and in \s\ref{ss:perm}) we use the symbols $A$, $\wt W$, and $X_*$ to denote objects attached to an arbitrary based root datum; in all other subsections we shall resume using these symbols for their original meanings.  We take as our main references the papers of Steinberg \cite{st68}, Kottwitz and Rapoport \cite{kottrap00}, and Haines and Ng\^o \cite{hngo02a}, especially \cite{hngo02a}*{\s9}.

Let us briefly recall what we need from the theory of Steinberg fixed-point root data.  Let $\R = (X^*,X_*,R,R^\vee,\Pi)$ be a reduced and irreducible based root datum.  Attached to \R are its Weyl group $W$, its affine Weyl group $W_\aff$, and its extended affine Weyl group $\wt W := X_* \rtimes W$.  The simple roots $\Pi$ determine a distinguished alcove $A$ in the apartment $X_*\tensor_\ZZ \RR$, namely the unique alcove contained in the positive Weyl chamber and whose closure contains the origin.  The extended affine Weyl group then admits a second semidirect product decomposition $\wt W = W_\aff \rtimes \Omega$, where $\Omega$ is the stabilizer in $\wt W$ of $A$.  The length function $\ell$ on $W_\aff$ determined by $A$ extends to $\wt W$ via the rule $\ell(wx) = \ell(w)$ for $w \in W_\aff$ and $x \in \Omega$.

An \emph{automorphism} $\Theta$ of $\R$ is an automorphism $\Theta$ of the abelian group $X_*$ such that the subsets $R$, $\Pi \subset X^*$ are stable under the dual automorphism $\Theta^*$ of $X^*$ induced by $\Theta$ and the perfect pairing $X^* \times X_* \to \ZZ$.  It follows that any automorphism of $\R$ induces automorphisms of $W$, $W_\aff$, and $\wt W$, and that these induced automorphisms preserve the length functions on these groups.  Attached to $\Theta$ is the \emph{Steinberg fixed-point root datum} $\R^{[\Theta]} = (X^{*[\Theta]}, X_*^{[\Theta]}, R^{[\Theta]}, R^{\vee [\Theta]}, \Pi^{[\Theta]})$; this is a reduced and irreducible based root datum described explicitly in \cite{hngo02a}*{\s9}.  We systematically use a superscript $[\Theta]$ to denote the analogs for $\R^{[\Theta]}$ of all the objects defined for $\R$.  For our purposes, we shall just mention that $\wt W^{[\Theta]}$ is naturally a subgroup of $\wt W$, with $W_\aff^{[\Theta]}$ equal to the fixed-point subgroup $W^\Theta_\aff \subset W_\aff$ and $\Omega^{[\Theta]}$ equal to the intersection $\wt W^{[\Theta]} \cap \Omega$.

%

\begin{lem}\label{st:fprd_lem}
Let $J$ and $J'$ be $\Theta$-stable subsets of simple reflections in $W_\aff$, and let $W_J$ and $W_{J'}$ denote the respective subgroups they generate.  
Suppose that $w \in W_J\bs \wt W / W_{J'}$ is a double coset that meets $\wt W^{[\Theta]}$, and let $\wt w_0$ denote the unique representative of minimal length in $\wt W$ of $w$.  Then $\wt w_0 \in \wt W^{[\Theta]}$.
\end{lem}

\begin{proof}
Say that $\wt w_0 = w_0 x_0$ with $w_0 \in W_\aff$ and $x_0 \in \Omega$.  We'll show that $w_0 \in W_\aff^{[\Theta]}$ and $x_0 \in \Omega^{[\Theta]}$.

By assumption $W_J \wt w_0 W_{J'} \cap \wt W^{[\Theta]}$ contains some element $\wt w$; say $\wt w = wx$ with $w \in W_\aff^{\Theta}$ and $x \in \Omega^{[\Theta]}$.  Since $\wt w_0$ and $\wt w$ are evidently congruent mod $W_\aff$, we have $x_0 = x$ and $x_0 \in \Omega^{[\Theta]}$.

Now observe that
\[
   W_J \wt w_0 W_{J'} = W_J w_0 x_0 W_{J'} = W_J w_0 W_{\smash[t]{x_0J'x_0^{-1}}}x_0.
\]
Since $x_0 \in \Omega^{[\Theta]}$, the set $x_0J'x_0^{-1}$ is again $\Theta$-stable.  Since $w$ is $\Theta$-fixed, the double cosets $W_J w_0 W_{x_0J'x_0^{-1}}$ and $(W_J w_0 W_{x_0J'x_0^{-1}})^\Theta = W_J w_0^\Theta W_{x_0J'x_0^{-1}}$ both contain $w$, hence are equal.  But $\ell(w_0) = \ell(w_0^\Theta)$, and to say that $\wt w_0$ is of minimal length in $W_J \wt w_0 W_{J'}$ is precisely to say that $w_0$ is of minimal length in $W_J w_0 W_{x_0J'x_0^{-1}}$.  Hence $w_0 = w_0^\Theta$ by uniqueness of the representative of minimal length, and $w_0 \in W_\aff^{\Theta}$, as desired.
\end{proof}

\subsection{A theorem of Haines and Ng\^o}\label{ss:hngo}
We return to our earlier notation for $X_*$ and $\wt W$, where these respectively denote the cocharacter lattice and extended affine Weyl group for the root datum of $GSp_{2m}$.  This root datum is a Steinberg fixed-point root datum obtained from the root datum for $GL_{2m}$, and we'll need the Iwahori-Weyl group $\wt W_{GL_{2m}} := \ZZ^{2m} \rtimes S_{2m}$ and the natural embedding $\wt W \inj \wt W_{GL_{2m}}$ for the theorem of Haines and Ng\^o.

Actually, it will be convenient for us to split their theorem into two parts.%
\footnote{Though the way in which we shall do so is not reflective of how they \emph{prove} the theorem.}
For both parts, we'll need the alcove for $GL_{2m}$ in $\RR^{2m}$ determined by the vectors $\omega_i$ for $i \in \ZZ$.  We denote by $W_{GL_{2m},\pm I}$ the stabilizer in $\wt W_{GL_{2m}}$ of all the $\omega_i$ for $i \in 2m\ZZ \pm I$.  Then $W_{GL_{2m},\pm I}$ is generated by the reflections across the walls of the base alcove that contain all the $\omega_i$ for $i \in 2m\ZZ \pm I$; and $W_I = \wt W \cap W_{GL_{2m},\pm I}$, so that $\wt W/W_I \inj \wt W_{GL_{2m}}/ W_{GL_{2m},\pm I}$.

Our base alcove determines a Bruhat order on $\wt W_{GL_{2m}}$, and for any cocharacter $\mu \in \ZZ^{2m}$, we let $\Adm_{GL_{2m},\pm I}(\mu)$ denote the set of all $\mu$-admissible elements in $\wt W_{GL_{2m}}/W_{GL_{2m},\pm I}$.  When $\mu \in X_*$, we may also consider the set $\Adm_{GSp_{2m},I}(\mu)$ of all $\mu$-admissible elements in $\wt W/W_I$.  In the Iwahori case $I = \{0,\dotsc,m\}$, we write just $\Adm_{GL_{2m}}(\mu)$ and $\Adm(\mu)_{GSp_{2m}}$ in place of $\Adm_{GL_{2m},\pm\{0,\dotsc,m\}}(\mu)$ and $\Adm_{GSp_{2m},\{0,\dotsc,m\}}(\mu)$, respectively.  We now have the first part of the theorem of Haines and Ng\^o, generalized to the general parahoric case.  

\begin{thm}\label{st:adm_intersect}
Let $\mu\in X_*$ be any cocharacter for $GSp_{2m}$.  Then
\[
   \Adm_{GSp_{2m},I}(\mu) = \Adm_{GL_{2m},\pm I}(\mu) \cap \wt W/W_I.
\]
\end{thm}

\begin{proof}
In the Iwahori case, this is just the combination of Theorem 1 and Proposition 5 in \cite{hngo02a}.  We deduce the general case from this and \eqref{st:fprd_lem}.  If $w \in \wt W/W_I$ is $\mu$-admissible, then choose a lift $\wt w \in \wt W$ which is $\mu$-admissible.  By a theorem of Kottwitz and Rapoport \cite{kottrap00}*{1.8, 2.3}, $\wt W$ inherits its Bruhat order from $\wt W_{GL_{2m}}$, whence $\wt w \in \Adm_{GL_{2m}}(\mu) \cap \wt W$ and $w \in \Adm_{GL_{2m},\pm I}(\mu) \cap \wt W/W_I$.  Conversely, suppose $w \in \Adm_{GL_{2m},\pm I}(\mu) \cap \wt W /W_I$.  Let $\wt w_0 \in \wt W_{GL_{2m}}$ denote the minimal length representative of $w$.  Then $\wt w_0$ is contained in $\wt W$ by \eqref{st:fprd_lem} and in $\Adm_{GL_{2m}}(\mu)$ by the analog of \eqref{rk:bo_double_cosets} for $\wt W_{GL_{2m}}$.  Hence $\wt w_0 \in \Adm_{GSp_{2m}}(\mu)$ by the Iwahori case of the theorem.  Hence $w \in \Adm_{GSp_{2m},I}(\mu)$.
\end{proof}

For applications we'll need the second part of Haines and Ng\^o's theorem, which replaces $\Adm_{GL_{2m},\pm I}(\mu)$ with the set of \emph{$\mu$-permissible} elements
\[
   \Perm_{GL_{2m},\pm I}(\mu) :=
   \biggl\{\, \wt w \in \wt W_{GL_{2m}}/W_{GL_{2m},\pm I} \biggm| 
      \begin{varwidth}{\textwidth}
         \centering
         $\wt w \cdot \omega_i - \omega_i \in \Conv(S_{2m}\mu)$\\
         for all $i \in 2m\ZZ \pm I$
      \end{varwidth}
   \,\biggr\},
\]
where $\Conv(S_{2m}\mu)$ is the convex hull in $\RR^{2m}$ of the Weyl orbit of $\mu$.    By \cite{hngo02a}*{Theorem 1} in the Iwahori case and G\"ortz's generalization \cite{goertz05}*{Corollary 9} to the general parahoric case, one has
\[
   \Perm_{GL_{2m},\pm I}(\mu) = \Adm_{GL_{2m},\pm I}(\mu)
\]
for any $\mu \in \ZZ^{2m}$.  We get the second part of the theorem simply by plugging this in to \eqref{st:adm_intersect}.

\begin{thm}\label{st:perm_intersect}
Let $\mu\in X_*$ be any cocharacter for $GSp_{2m}$.  Then
\begin{flalign*}
\phantom{\qed} & & \Adm_{GSp_{2m},I}(\mu) = \Perm_{GL_{2m},\pm I}(\mu) \cap \wt W/W_I. & & \qed
\end{flalign*}
\end{thm}

\begin{rk}
We likewise obtain obvious double coset versions of \eqref{st:adm_intersect} and \eqref{st:perm_intersect}, where one has equalities between subsets of $W_I \bs \wt W / W_I$, which were anticipated in \cite[Notes added June 2003, no.\ 3]{rap05} (note that the intersection as written in \cite{rap05} should be with $\wt W^K\bs \wt W(GSp)/ \wt W^K$ in place of $\wt W(GSp)$).
\end{rk}

\subsection{Wedge-permissibility and \texorpdfstring{$\mu$}{mu}-admissibility}\label{ss:wedge-perm_mu-adm}
In this subsection we prove \eqref{st:wedge_equiv_adm}, and with it complete the proof of \eqref{st:adm<=>perm_I}.  We shall do so by applying \eqref{st:perm_intersect}, for which we need a good description of the convex hull $\Conv(S_{2m}\mu)$.  Let $\nu_i := \bigl( 1^{(i)}, 0^{(2m-i)} \bigr)$ for $1 \leq i \leq 2m$.  The following lemma is certainly well-known, but for convenience we give a proof.

\begin{lem}\label{st:GL_conv_hull}
For any dominant cocharacter $\mu = (n_1,\dotsc,n_{2m}) \in \ZZ^{2m}$, we have
\begin{align*}
   \Conv(S_{2m}\mu)
      &= \biggl\{\, x \in \RR^{2m} \biggm| 
              \begin{varwidth}{\textwidth}
                 \centering
                 $\nu \cdot x \leq n_1 + \dotsb + n_i$ for all $1 \leq i \leq 2m$ and\\
                 all $\nu \in S_{2m}\nu_i$, with equality when $i = 2m$
              \end{varwidth}
              \,\biggr\}\\
      &= \biggl\{\, x \in \RR^{2m} \biggm| 
              \begin{varwidth}{\textwidth}
                 \centering
                 $n_{2m+1-i} + \dotsb + n_{2m} \leq \nu \cdot x$ for all $1 \leq i \leq 2m$\\ and all $\nu \in S_{2m}\nu_i$, with equality when $i = 2m$
              \end{varwidth}
              \,\biggr\}.
\end{align*}
\end{lem}

Here we mean dominant in the usual sense for cocharacters of the standard maximal torus in $GL_{2m}$, namely $n_1 \geq \dotsb \geq n_{2m}$; and by $\nu \cdot x$ we mean the usual dot product of vectors in $\RR^{2m}$.

\begin{proof}[Proof of \eqref{st:GL_conv_hull}]
The second equality is trivial; we prove the first.  Let $S$ denote the set appearing on the right-hand side of the first asserted equality.  Then $S$ is plainly convex and contains $S_{2m}\mu$.  Hence $S$ contains $\Conv(S_{2m}\mu)$.  To check the reverse inclusion, we use that $\Conv(S_{2m}\mu)$ consists precisely of the vectors $x$ such that, for all $\sigma \in S_{2m}$, $\mu - \sigma x$ is a nonnegative linear combination of positive $GL_{2m}$-coroots.  Let $x = (x_1,\dotsc,x_{2m}) \in S$.  Since $S$ is plainly $S_{2m}$-stable, it suffices to show that just $\mu -x$ is a nonnegative linear combination of positive coroots.  By definition of $S$, we have $x_{2m} - n_{2m} = n_1 + \dotsb + n_{2m-1} - x_1 - \dotsb - x_{2m-1}$.  Hence, letting $e_1,\dotsc$, $e_{2m}$ denote the standard basis in $\RR^{2m}$, we have
\begin{multline*}
   \mu - x = (n_1-x_1)(e_1-e_2) + (n_1 + n_2 - x_1 -x_2)(e_2-e_3)\\
              + \dotsb + (n_1 + \dotsb + n_{2m-1} - x_1 - \dotsb - x_{2m-1})(e_{2m-1} - e_{2m}),
\end{multline*}
which is of the desired form.  
\end{proof}

\begin{proof}[Proof of \eqref{st:wedge_equiv_adm}]
Everything is now transparent:  the lemma makes it obvious that the set of wedge-permissible elements in $\wt W/W_I$ equals $\Perm_{GL_{2m},\pm I}(\mu) \cap \wt W/W_I$ for $\mu = \bigl(2^{(s)},1^{(2m-2s)},0^{(s)}\bigr)$, and we then apply \eqref{st:perm_intersect}.
\end{proof}

The proof of \eqref{st:adm<=>perm_I} is now complete.

\subsection{Permissibility}\label{ss:perm}
Our aim in the final two subsections of the paper is to show that the notions of $\mu$-admissibility and $\mu$-permissibility coincide for the cocharacter $\mu = \bigl(2^{(s)},1^{(2m-2s)},0^{(s)}\bigr)$ of $GSp_{2m}$, $0 \leq s \leq m$.  In this subsection we return to the general setup and notation of \s\ref{ss:steinberg_fprd}.

Let $F$ be a facet of the base alcove $A$, let $J_F$ be the set of simple reflections across the walls of $A$ that contain $F$, and let $W_{J_F}$ denote the subgroup of $W_\aff$ generated by $J_F$.  Let $\mu \in X_*$ be any cocharacter. We write $t_\mu$ when we wish to regard $\mu$ as an element in $\wt W$.  Recall that an element $w \in W_{J_F} \bs \wt W / W_{J_F}$ is \emph{$\mu$-permissible} if $w \equiv t_\mu \bmod W_\aff$ and $w \cdot a - a \in \Conv(W\mu)$ for all $a \in F$, where $\Conv(W\mu)$ denotes the convex hull of the Weyl orbit $W\mu$ in $X_*\tensor_\ZZ \RR$; this condition is well-defined on the double coset $w$ as shown by Rapoport \cite{rap05}*{\s 3}.  Note that the containment $w\cdot a - a \in \Conv(W\mu)$ holds for all $a \in F$ $\iff$ for each subfacet $F'$ of $F$ of minimal dimension, the containment $w\cdot a -a \in \Conv(W\mu)$ holds for some $a \in F'$.  The notion of $\mu$-permissibility for elements $w \in \wt W/W_{J_F}$ is defined in an entirely analogous way.

It is known from examples of Haines and Ng\^o \cite{hngo02a}*{Theorem 3} that $\mu$-permissi\-bil\-ity is not well-behaved with regard to Steinberg fixed-point root data, in the sense that \eqref{st:adm_intersect} no longer holds in general when we replace the admissible sets on both sides by the corresponding permissible sets.  Nevertheless, the convex hulls that come up are at least well-behaved, as we now show.  Recall from \cite{hngo02a}*{\s9} that the cocharacter lattice $X_*^{[\Theta]}$ of $\R^{[\Theta]}$ is the subgroup of $X_*$
\[
   X_*^{[\Theta]} := \bigl\{\, x \in X_* \bigm| \Theta(x) \equiv x \bmod Z \,\bigr\},
\]
where $Z := \{\, x \in X_* \mid \langle \alpha, x \rangle = 0\ \text{for all}\ \alpha \in R\,\}$.  Let $V := X_* \tensor \RR$ and $V^{[\Theta]} := X_*^{[\Theta]} \tensor \RR$.

\begin{lem}\label{st:conv_fpdr}
Let $\mu \in X_*^{[\Theta]}$.  Then $\Conv(W^{[\Theta]} \mu) = \Conv(W\mu) \cap V^{[\Theta]}$.
\end{lem}

\begin{proof}
The containment $\subset$ holds since $\Conv(W\mu) \cap V^{[\Theta]}$ is convex and contains $W^{[\Theta]} \mu$.  To establish the containment $\supset$, similarly to the proof of \eqref{st:GL_conv_hull}, we use that $\Conv(W^{[\Theta]}\mu)$ consists precisely of the elements $x \in V^{[\Theta]}$ such that, for all $\sigma \in W^{[\Theta]}$, $\mu - \sigma x$ is a nonnegative linear combination of positive $\R^{[\Theta]}$-coroots.  So let $x \in \Conv(W\mu) \cap V^{[\Theta]}$ and $\sigma \in W^{[\Theta]}$.  By the analogous statement for $\Conv(W\mu)$, $\mu - \sigma x$ is expressible as a nonnegative linear combination of positive $\R$-coroots,
\[
   \mu - \sigma x = \sum_{\alpha^\vee \in R^\vee_+} c_{\alpha^\vee} \alpha^\vee.
\]
Hence $\mu - \sigma x \in V^{[\Theta]} \cap (Q^\vee \tensor \RR) = (Q^{\vee})^{\Theta} \tensor \RR$.  Hence $\Theta(\mu - \sigma x) = \mu - \sigma x$.  Hence for any positive integer $N$,
\[
   \mu - \sigma x = \frac 1 N \sum_{i=0}^{N-1}\Theta^i(\mu-\sigma x)
           = \sum_{\alpha^\vee \in R^\vee_+} c_{\alpha^\vee} \frac 1 N
                  \sum_{i=0}^{N-1}\Theta^i(\alpha^\vee).
\]
It follows from the description of the coroots in \cite{hngo02a}*{\s9} that for $\alpha^\vee$ a positive $\R$-coroot and $N$ equal to the order of $\Theta|_{Q^\vee}$, $\frac 1 N \sum_{i=0}^{N-1}\Theta^i(\alpha^\vee)$ is a positive multiple of a positive $\R^{[\Theta]}$-coroot. The conclusion follows.
\end{proof}

\subsection{\texorpdfstring{$\mu$}{mu}-permissibility and \texorpdfstring{$\mu$}{mu}-admissibility}\label{ss:perm_adm}
We again return to our original notation for $X_*$ and $\wt W$.  We now conclude the paper by showing that $\mu$-admissibility and $\mu$-permissibility in $\wt W/W_I$ are equivalent in the case of the cocharacter $\mu = \bigl(2^{(s)},1^{(2m-2s)},0^{(s)}\bigr)$.  

To proceed we'll need the (again well-known) analog of \eqref{st:GL_conv_hull} for $GSp_{2m}$.  Let
\[
   V:= X_* \tensor \RR \ciso \bigl\{\,(x_1,\dotsc,x_{2m}) \in \RR^{2m} \bigm| x_1 + x_{2m} = \dots = x_m + x_{m+1}\,\bigr\}.
\]
For $x = (x_1,\dotsc,x_{2m}) \in V$, let $c(x)$ denote the common real number $x_1 + x_{2m} = \dots = x_m + x_{m+1}$.

\begin{lem}\label{st:GSp_conv_hull}
For any dominant cocharacter $\mu = (n_1,\dotsc,n_{2m}) \in X_*$, we have
\begin{align*}
   \Conv(S_{2m}^*\mu)
      &= \Conv(S_{2m}\mu) \cap V\\
      &= \biggl\{\, x \in V \biggm| 
               \begin{varwidth}{\textwidth}
                  \centering
                  $c(x) = c(\mu)$ and $\nu \cdot x \leq n_1 + \dotsb + n_i$ for\\
                  all $1 \leq i \leq m$ and all $\nu \in S_{2m}^*\nu_i$
               \end{varwidth}
            \,\biggr\}\\
      &= \biggl\{\, x \in V \biggm| 
               \begin{varwidth}{\textwidth}
                  \centering
                  $c(x) = c(\mu)$ and $n_{2m+1-i} + \dotsb + n_{2m} \leq \nu \cdot x$\\
                  for all $1 \leq i \leq m$ and all $\nu \in S_{2m}^*\nu_i$
               \end{varwidth}
            \,\biggr\}.
\end{align*}
\end{lem}

Here we again mean dominant in the usual sense, namely $n_1 \geq \dotsb \geq n_{2m}$.

\begin{proof}[Proof of \eqref{st:GSp_conv_hull}]
The first equality is just an application of \eqref{st:conv_fpdr}, and the other two follow easily from this and \eqref{st:GL_conv_hull} (or can be proved directly in a way entirely analogous to the proof of \eqref{st:GL_conv_hull}).
\end{proof}

For $i = 0,\dotsc,$ $m$, let
\[
   \eta_i :=\bigl((-\tfrac 1 2)^{(i)},0^{(2m-2i)},(\tfrac 1 2)^{(i)} \bigr) = \frac{\omega_i + \omega_{-i}}2.
\]
The points $\eta_0,\dotsc,$ $\eta_m$ serve as ``vertices'' for the base alcove \eqref{disp:GSp_alc}, in the sense that each facet of minimal dimension (namely $1$) contains exactly one of the $\eta_i$'s.

\begin{prop}\label{st:perm_adm}
Let $\mu = \bigl(2^{(s)},1^{(2m-2s)},0^{(s)}\bigr)$.  Then $\wt w \in \wt W/W_I$ is $\mu$-admissi\-ble $\iff$ $\wt w$ is $\mu$-permissible.
\end{prop}

\begin{proof}
We shall actually show that $\wt w$ is wedge-permissible $\iff$ $\wt w$ is $\mu$-permissible.  Let $\mu_i := \wt w \cdot \omega_i - \omega_i$ for $i \in 2m\ZZ \pm I$, and note that
\[
   \wt w \cdot \eta_i - \eta_i = \frac{\mu_i + \mu_{-i}}2
   \quad\text{for}\quad i \in I.
\]
The implication $\Longrightarrow$ is either obvious now from the definition of wedge-permissibility and \eqref{st:GSp_conv_hull}; or follows from the general fact due to Kottwitz and Rapoport \cite{kottrap00}*{11.2} that $\mu$-admissibility always implies $\mu$-permissibility for any cocharacter $\mu$ in any extended affine Weyl group.  To prove the implication $\Longleftarrow$, suppose that $\wt w$ is $\mu$-permissible, and let $i \in 2m\ZZ \pm I$.  Since $\frac{\mu_i + \mu_{-i}}2 \in \Conv(S_{2m}^*\mu)$, it is clear that $\mu_i$ (and $\mu_{-i}$) satisfy \eqref{it:naive_perm_2}.  Hence $\mu_i + \mu_{-i}^* = \mathbf 2$.  Hence
\[
   \frac{\mu_i + \mu_{-i}}2 = \frac{\mu_i + \mathbf 2 - \mu_i^*}2.
\]
Varying $\nu \in S_{2m}^*\nu_1$, we deduce from \eqref{st:GSp_conv_hull} that
\[
   \mathbf{-2} \leq \mu_i - \mu_i^* \leq \mathbf 2.
\]
Hence, by the obvious analog of the basic inequalities (\ref{st:basic_ineqs}, \ref{rk:basic_ineqs_arb_index}) for $\wt W/W_I$ (here with $d = 2$),
\[
   \mathbf 0 \leq \mu_i \leq \mathbf 2,
\]
and $\mu_i$ satisfies \eqref{it:naive_perm_1}.

To complete the proof, suppose by contradiction that $\mu_i$ does not satisfy \eqref{it:wedge_perm}.  More precisely, let
\begin{align*}
   E &:= \bigl\{\, j \in \{1,\dotsc,2m\} \bigm| \mu_i(j) = 0 \text{ and } \mu_i(j^*) = 2 \,\bigr\},\\
   F &:= \bigl\{\, j \in \{1,\dotsc,2m\} \bigm| \mu_i(j) = 0 \text{ and } \mu_i(j^*) = 1 \,\bigr\},\\
   G &:= \bigl\{\, j \in \{1,\dotsc,2m\} \bigm| \mu_i(j) = 1 \text{ and } \mu_i(j^*) = 2 \,\bigr\},\ \text{and}\\
   H &:= \bigl\{\, j \in \{1,\dotsc,2m\} \bigm| \mu_i(j) = 1 \text{ and } \mu_i(j^*) = 1 \,\bigr\},
\end{align*}
where $j^* := 2m+1-j$.  Then, by the basic inequalities,
\[
   \{1,\dotsc,2m\} = E \amalg F \amalg G \amalg H \amalg E^* \amalg F^* \amalg G^*,
\]
where we write $S^* := 2m+1-S$ for any subset $S\subset \{1,\dotsc,2m\}$.  Let
\[
   e:= \# E, \quad f:= \# F, \quad\text{and}\quad g := \# G.
\]
Our assumption to obtain a contradiction is that
\[
   e + f = \#(E \amalg F) = \#\bigl\{\, j \in \{1,\dotsc,2m\} \bigm| \mu_i(j) = 0 \,\bigr\} > s.
\]
Under our assumption, we may write $e + f = s + t$ with $t > 0$.  Since $\#(E \amalg F \amalg G) = e + f + g = s + t + g$, we conclude from \eqref{st:GSp_conv_hull} that
\[
   t+g
   \leq \sum_{j\in E \amalg F \amalg G}
      \biggl( \frac{\mu_i + \mu_{-i}}2\biggr)(j)
   = \sum_{j\in E \amalg F \amalg G}
      \frac{2+ \mu_i(j) - \mu_{i}(j^*)}2
   = \frac{f+g}2.
\]
But the equality $\Sigma \mu_i = 2m$ forces $f = g$.  Hence the last expression in the display equals $g$, a contradiction.
\end{proof}

\begin{rk}
In the case of local models for ramified $GU_n$ for $n$ for \emph{even,} we shall show in \cite{sm10a} --- via essentially the same proof --- that the Schubert cells contained in the special fiber of $M_I^\wedge$ inside the affine flag variety are indexed by a variant of the $\mu$-permissible set in which the requirement that $w$ and $t_\mu$ become equal in $\wt W/W_\aff$ is weakened to require only that $w$ and $t_\mu$ become equal in $\wt W/W_\aff$ mod torsion.  It turns out that neither variant agrees in general with the $\mu$-admissible set, which we shall show indexes the Schubert cells contained in the special fiber of $M_I^\spin$.
\end{rk}

%

%


%
%
%
%

%

%
%



\end{document}